\documentclass[11pt]{article}

\usepackage{setspace}

\usepackage{amsmath}
\usepackage{amssymb}
\usepackage{amsthm}
\usepackage{titlesec}
\usepackage{graphicx}
\usepackage{caption}
\usepackage{subcaption}

\usepackage{diagbox}

\usepackage[T1]{fontenc}
\usepackage[utf8]{inputenc}

\usepackage{comment}

\usepackage{indentfirst}
\usepackage[top=1.25in,left=1.25in,right=1.25in]{geometry}
\newlength{\numone}
\newlength{\widone}
\newlength{\numtwo}
\newlength{\widtwo}
\newtheorem{thm}{Theorem}[section]
\newtheorem{lemma}[thm]{Lemma}
\newtheorem{cor}[thm]{Corollary}
\newtheorem{defin}[thm]{Definition}
\newtheorem{que}[thm]{Question}
\newtheorem{alg}[thm]{Algorithm}
\newtheorem{rem}[thm]{Remark}

\numberwithin{equation}{section}

\author{\Large{Riccardo W. Maffucci}}
\newcommand{\Addresses}{  
	R.W.~Maffucci, \textsc{Dipartimento di Matematica, Universit\`a di Torino\\\indent Via Carlo Alberto 10, Turin 10123, Italy}\par\nopagebreak\vspace{-0.35cm}
	\textit{E-mail address}, R.W.~Maffucci: {\texttt{riccardowm@hotmail.com}}
}

\title{\Large{\uppercase{\bf On self-duality and unigraphicity for $3$-polytopes}}}

\date{}

\def\calE{\mathcal{E}}
\def\calF{\mathcal{F}}

\def\calP{\mathcal{P}}

\def\calV{\mathcal{V}}

\def\calZ{\mathcal{Z}}

\begin{document}
\titleformat{\section}
  {\Large\scshape}{\thesection}{1em}{}
\titleformat{\subsection}
  {\large\scshape}{\thesubsection}{1em}{}
\maketitle
\Addresses


\begin{abstract}
Recent literature \cite{maffucci2024characterising,delmaf,maffucci2025faces} investigated the problem of characterising the graph degree sequences with exactly one $3$-polytopal (planar, $3$-connected) realisation. This seems to be a difficult problem in full generality. In this paper, we characterise the sequences with exactly one self-dual $3$-polytopal realisation.

To settle our question and construct the relevant graphs, we apply iterated graph transformations. One of these is a generalisation of an algorithm \cite{mafpo3} that constructs a self-dual $3$-polytope for any degree sequence such that at least one such realisation exists.
\end{abstract}

\section{Introduction}
\subsection{Unigraphicity}
In this paper, we will work with undirected graphs $G$ having no multiple edges or loops. The degree sequence of $G$ is
\[\sigma: \ d_1,d_2,\dots,d_p,\]
where $V(G)=\{v_1,v_2,\dots,v_p\}$ is the set of vertices, and
\[d_i=\deg(v_i), \quad 1\leq i\leq p.\]
The degree sequence is of course defined up to reordering. The graph $G$ is called a \textit{realisation} of $\sigma$. In the other direction, Havel \cite{have55}, Hakimi \cite{hakimi}, and Erd\"os-Gallai \cite{erdgal} characterised the sequences of non-negative integers that have a graph realisation.

An interesting question is, does $\sigma$ uniquely determine $G$? In other words, which sequences have a unique graph realisation? For instance, \[2,2,2,1,1\]
is realised by $\calP_5$, the elementary path on $5$ vertices, but also by $K_3\cup K_2$, where $K_p$ denotes the complete graph of order $p$, and $\cup$ a disjoint union of graphs. On the other hand, $2^5$ (where the notation $n^k$ stands for the value $n$ being repeated $k$ times) has a unique realisation as a pentagon ($5$-cycle). We call these sequences (and corresponding realisations) \textbf{unigraphic}. This theory has been investigated e.g.\ in \cite{koren1,li1975,john80}. For more on degree sequences, we refer the interested reader to \cite{bose08,raosur,tysh87} and to the survey \cite{haki06}.

\subsection{Self-duality}
The study of polyhedra began as early as antiquity, with the five Platonic solids (regular polyhedra). These were of importance, to the extent that Plato related each to one of the five `natural elements' (e.g., the octahedron represented air). Euclid described their construction, and was aware that there are no other regular polyhedra. It was already known to the ancient Greeks that the tetrahedron is self-dual, the cube and octahedron form a dual pair, and the dodecahedron and icosahedron form a dual pair (see below for more details). Much later, Kepler assigned a regular polyhedron to each of the other planets in the Solar System observed at the time.

Nowadays we know that a graph is the $1$-skeleton of a polyhedral solid if and only if it is $3$-connected and planar -- the Rademacher-Steinitz Theorem \cite{radste}. Homeomorphic polyhedra have isomorphic graphs. We will refer to these as $3$-polytopal or polyhedral graphs interchangeably.

A graph is planar if it may be embedded in the plane (equivalently, on the surface of a sphere) such that no edges cross except at vertices. A plane graph is a planar graph considered together with an embedding in the plane. Whitney observed that a $3$-connected, planar graph has a unique embedding in the surface of a sphere, and thus a unique embedding in the plane once the external region has been chosen \cite{whit32}. For this reason, one may talk about {\em the} planar immersion of a polyhedron. Polyhedral vertices correspond to graph vertices, edges to edges, and polyhedral faces to plane graph regions in a natural way. For example, see \cite[Appendix A]{mafpo1} for graphical representations of all $3$-polytopes with $14$ or fewer edges. Tutte \cite{tutt61} described an algorithm to construct all polyhedra on $q$ edges starting from the set of polyhedra on $q-1$ edges. This algorithm relies also on constructing \textit{duals}.

The dual $G^*$ of a plane graph $G$ is obtained by defining a vertex of $G^*$ for each region of $G$, and an edge in $G^*$ between each pair of vertices corresponding to regions of $G$ that share an edge (`adjacent regions'). In general, the dual of a plane graph may contain loops and/or multiple edges and moreover, given a planar graph $G$, possibly more than one plane dual may be constructed, depending on the embedding of $G$. It is a remarkable property of polyhedral graphs that their duals are not only still free from multiple edges and loops, but also still polyhedral. Further, as polyhedra have a unique embedding up to choosing the external region, the dual is unique up to isomorphism. For this reason, one may use the terminology `{\em the} dual polyhedron'. For instance, the $n$-gonal prism is the dual of the $n$-gonal bipyramid. The operation of constructing the dual polyhedron is an involution, in the sense that $(G^*)^*\simeq G$. A recent result on degree sequences for dual pairs appeared in \cite{boro20}.

Some polyhedra, such as the tetrahedron, or indeed any pyramid (wheel graph), are isomorphic to their dual, $G^*\simeq G$. They are called self-dual polyhedra. Their construction, and more generally of all self-dual planar maps, was achieved in \cite{arcric,serchr,serser}. If a graph degree sequence has at least one self-dual polyhedral realisation (sometimes referred to in the literature as being a `potentially self-dual polyhedral sequence') then it is of the form
\begin{equation}
	\label{eqn:seq}
	t_1,t_2,\dots,t_{k},3^m,
\end{equation}
where $T=(t_1,t_2,\dots,t_k)$ is a $k$-tuple of (not necessarily distinct) integers satisfying $t_i\geq 4$, $1\leq i\leq k$, and
\begin{equation}
	\label{eqn:m}
	m=m(T)=4+\sum_{i=1}^{k}(t_i-4)
\end{equation}
(by Euler's formula and the handshaking lemma). In the other direction, in \cite[Algorithm 3.2]{mafpo3} we described iterated graph transformations to construct, for each sequence of this type, a self-dual polyhedral realisation
\begin{equation}
	\label{eq:pt}
P(T)=P((t_1,t_2,\dots,t_k)).
\end{equation}
We will summarise this construction in section \ref{sec:pre}, as it constitutes a central part of our discussion. The interested reader may also refer to \cite{mafpo3}. The graph operations are of similar flavour to those of \cite{br2005} (see also \cite{bata89}). For algorithms to generate other classes of polyhedral graphs see e.g.\ \cite{hash11, dillen, bowe67,mafpo2,hollowbread2025generation,maffucci2024classification,de2024cancellation,maffucci2025regularity}, and for planar graphs see e.g.\ \cite{brin07,hare98,chro95}.

A natural question is, which sequences are unigraphic and polyhedral? Maybe surprisingly, there are only eight solutions to this problem \cite{maffucci2024rao} (cf. \cite{rao978}).

One might then ask the following.
\begin{que}
\label{que:gen}
	Among the sequences with at least one polyhedral realisation (potentially polyhedral), which ones have \textit{exactly one polyhedral realisation}?
\end{que}
These graphs are the so called `unigraphic polyhedra': they are completely determined by their degree sequence. For instance (cf. \cite[Figure 8]{mafpo1}), 
\[
5,4,4,3,3,3
\]
has the unique polyhedral realisation of Figure \ref{fig:u}, whereas the two polyhedra in Figure \ref{fig:s} share the same sequence $4,3^6$. This seems to be a difficult problem in full generality. Partial answers were achieved in \cite{maffucci2024characterising,delmaf}. A big step forward came with \cite{maffucci2025faces}, where we proved that the faces of unigraphic polyhedra (other than pyramids) have length at most $9$.
\begin{figure}[h!]
	\centering
	\begin{subfigure}[b]{0.4\textwidth}
		\centering
		\includegraphics[width=2.5cm,clip=false]{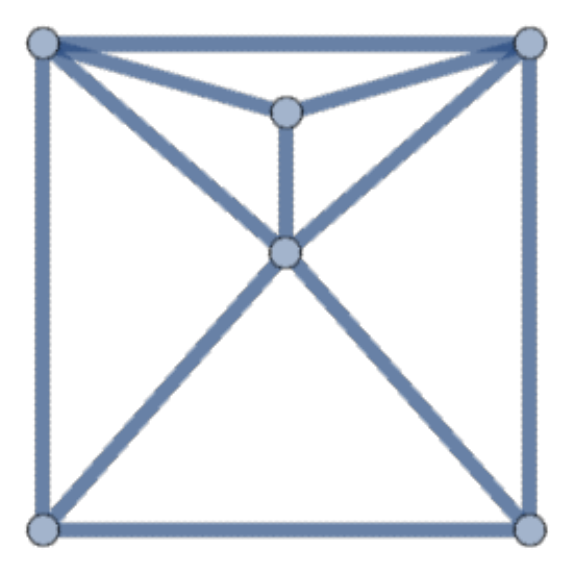}
		\caption{The only $3$-polytope of sequence $5,4,4,3,3,3$.}
		\label{fig:u}
	\end{subfigure}
	\hspace{0.5cm}
	\begin{subfigure}[b]{0.5\textwidth}
		\centering
		\includegraphics[width=2.5cm,clip=false]{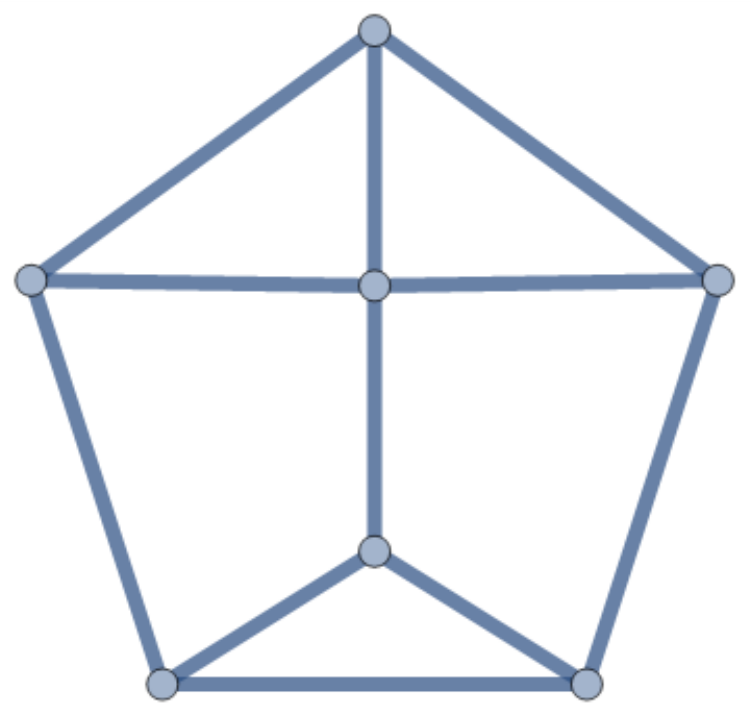}
		\hspace{0.25cm}
		\includegraphics[width=2.5cm,clip=false]{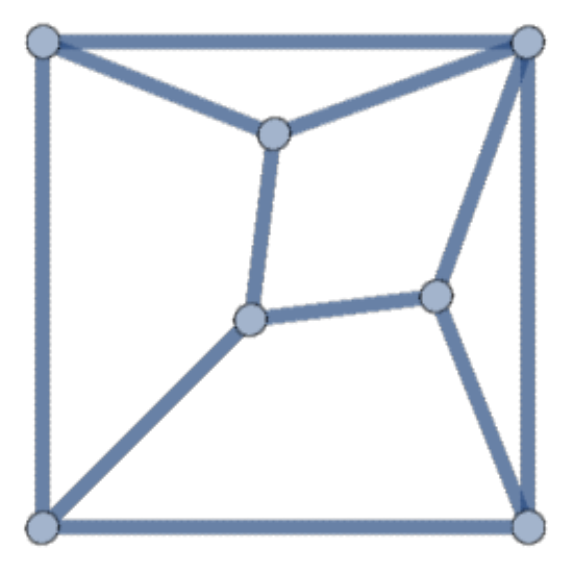}
		\caption{Two $3$-polytopes of sequence $4,3^6$.}
		\label{fig:s}
	\end{subfigure}
	\caption{}
	\label{fig:us}
\end{figure}

\subsection{The problem: sequences with a unique self-dual polyhedral realisation}
The main focus of this paper are the interactions between self-duality and unigraphicity. 
\begin{que}
Among the sequences \eqref{eqn:seq}, which ones are realised by exactly one self-dual polyhedron?
\end{que}
Of course \eqref{eqn:seq} may be realised by more than one self-dual $3$-polytope, and moreover, a realisation of \eqref{eqn:seq} is not necessarily a self-dual $3$-polytope. The result of our investigation is the following.

\begin{thm}
	\label{thm:1}
The only graph degree sequences with exactly one self-dual polyhedral realisation are
\begin{equation}
	\label{eqn:sxy}
x,y,3^{x+y-4}, \qquad x\geq y\geq 3.
\end{equation}
\end{thm}

The proof of Theorem \ref{thm:1} is written in section \ref{sec:proof}.

\begin{defin}
	\label{def:S}
We will use the notation $S(x,y)$ for the unique self-dual polyhedron realising \eqref{eqn:sxy}.
\end{defin}

\begin{figure}[h!]
	\centering
	\begin{subfigure}[b]{0.4\textwidth}
		\centering
		\includegraphics[width=4.0cm,clip=false]{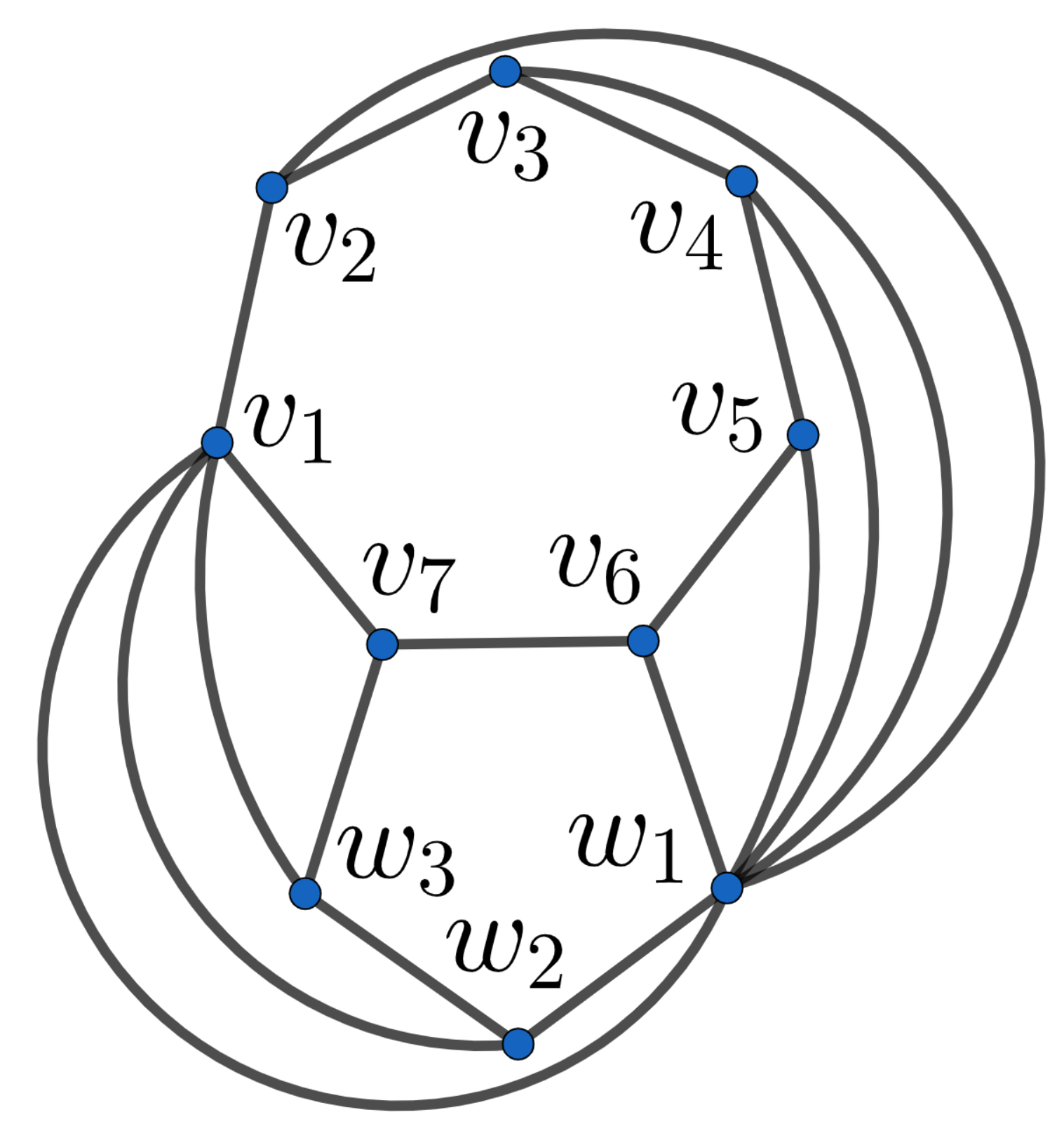}
		\caption{The self-dual $3$-polytope $S(7,5)$.}
		\label{fig:75}
	\end{subfigure}
	\hspace{0.5cm}
	\begin{subfigure}[b]{0.4\textwidth}
		\centering
		\includegraphics[width=4.5cm,clip=false]{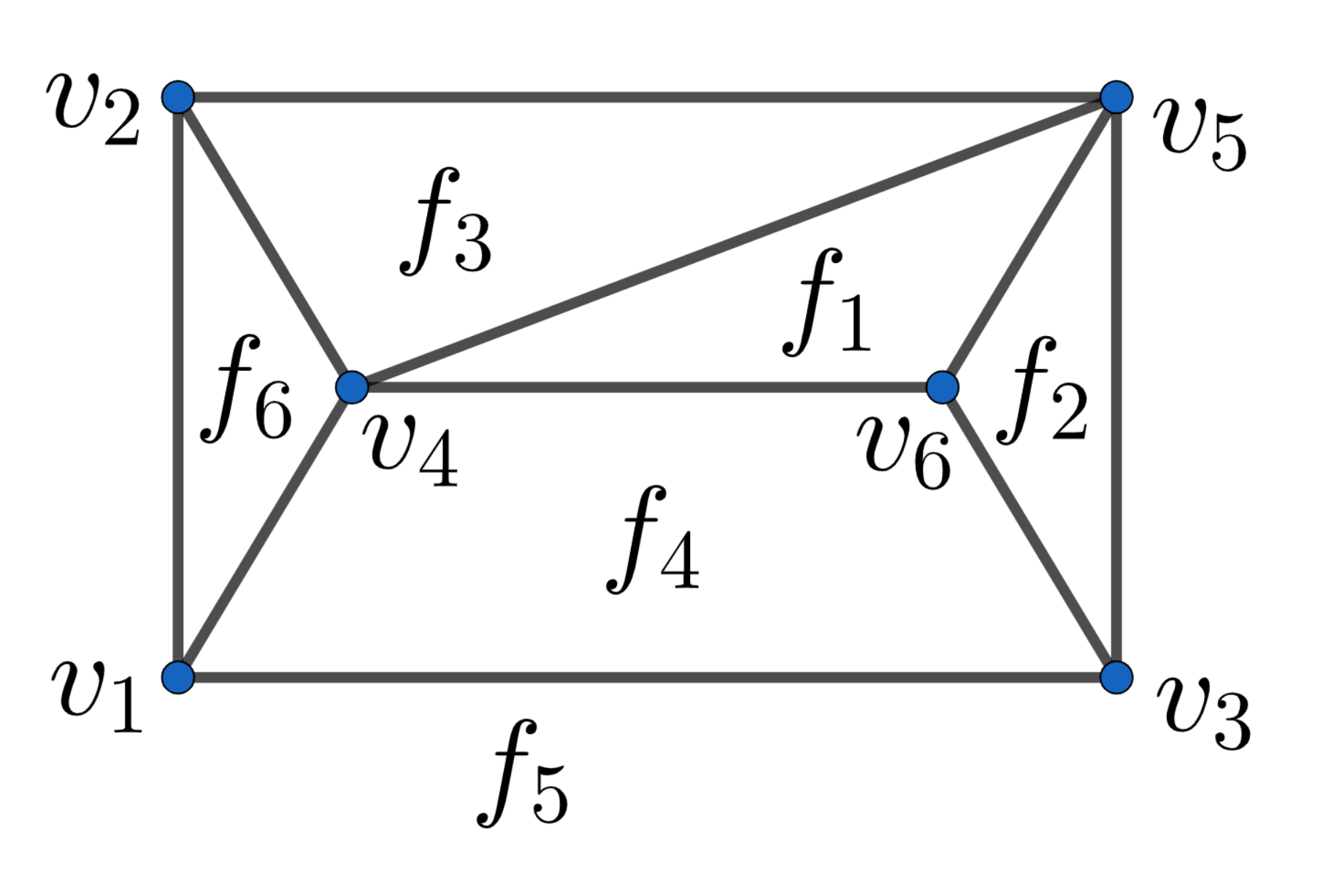}
		\caption{The self-dual $3$-polytope $S(4,4)$.}
		\label{fig:44}
	\end{subfigure}
\caption{Two examples of $S(x,y)$.}
\label{fig:7544}
\end{figure}

A couple of examples are given in Figure \ref{fig:7544}. As a consequence of Theorem \ref{thm:1}, each $S(x,y)$ with $y\geq 4$ may be constructed by inputting the pair $(x,y)$, in either order, into \cite[Algorithm 3.2]{mafpo3}. That is to say, referring to \eqref{eq:pt},
\[S(x,y)\simeq P((x,y))\simeq P((y,x)), \qquad x\geq y\geq 4.\]
Note that $S(x,3)$ is the $x$-gonal pyramid (wheel graph). For $x\geq 4$, we have $S(x,3)\simeq P((x))$.

We can also answer the following related question: among the unigraphic polyhedra, which ones are self-dual?
\begin{cor}
	\label{cor}
Among the graph degree sequences with exactly one polyhedral realisation, this realisation is self-dual only in the case of pyramids and of $S(4,4)$.
\end{cor}
Corollary \ref{cor} will be proven in section \ref{sec:1p}.

We also observe that, among the sequences with exactly one realisation in the class of all graphs, out of the eight $3$-polytopal solutions mentioned above from \cite{maffucci2024rao}, the only three self-dual polyhedra correspond to the $n$-gonal pyramids, for $n=3,4,5$.


\subsection{Plan of the paper}
The rest of the paper is organised as follows. In section \ref{sec:limi} we will collect prior results, as well as basic notation, terminology and conventions that will be used throughout. Section \ref{sec:proof} is entirely dedicated to proving Theorem \ref{thm:1}.

More in detail, sections \ref{sec:limi} and \ref{sec:proof} have the following structure. Section \ref{sec:bas} contains some basic graph-theoretical notation. Section \ref{sec:rad} presents the definition and a few properties of radial graphs. In section \ref{sec:H}, we will define certain subgraphs that will be employed later to show non-isomorphism between graphs. Section \ref{sec:Z} contains the definition of a graph transformation called $\calZ$ (in \cite{mafpo3} it is called $\calP$), that acts on the radial graph of a graph. Having defined $\calZ$, it will be possible to describe \cite[Algorithm 3.2]{mafpo3} in section \ref{sec:pre}. In section \ref{sec:deep} we will analyse the strengths of \cite[Algorithm 3.2]{mafpo3}.

Section \ref{sec:out} contains an outline of proof for Theorem \ref{thm:1}. In section \ref{sec:1p}, we will show that the degree sequence $x,y,3^{x-y-4}$ is unigraphic in the class of self-dual polyhedra. In section \ref{sec:2p}, we will prove non-unigraphicity for the sequence $t_1,\dots,t_k,3^m$ where $k\geq 3$ and the $t_i$'s are not all equal. In section \ref{sec:3p}, we will prove non-unigraphicity for $n^k,3^{4+k(n-4)}$ where $n\geq 5$ and $k\geq 3$. It then remains to show that $4^k,3^4$, $k\geq 3$ is not unigraphic, in section \ref{sec:4p}.



\section{Preliminaries}
\label{sec:limi}

Here we will collect prior results, notation, and terminology.

\subsection{Basic notation and terminology}
\label{sec:bas}
We will denote by $V(G)$ and $E(G)$ the vertex and edge sets of a graph $G$. The notation $\calP_n$ will indicate the elementary path on $n$ vertices, and $\cup$ the disjoint union of graphs. Regions of a plane graph are written in the form
\[F=[u_1,u_2,\dots,u_n].\]

The notation
\[\sigma: \ d_1,d_2,\dots,d_p,\]
refers to the degree sequence of $G$, where $V(G)=\{v_1,v_2,\dots,v_p\}$ and
\[d_i=\deg(v_i), \quad 1\leq i\leq p.\]
Repeated values in $\sigma$ are indicated with exponents i.e., $n^k$ means that the value $n$ is repeated $k$ times.

\subsection{Radial graphs}
\label{sec:rad}
We will construct self-dual $3$-polytopes of desired degree sequence via their \textit{radial graph}, also known as the vertex-face graph. Given the plane graph $\Gamma=(\calV,\calE)$ together with its set of regions $\calF$, the radial graph is defined as
\[R(\Gamma)=(\calV\cup\calF, E(R(\Gamma))),\]
where
\[E(R(\Gamma))=\{vf : v\in\calV, f\in\calF, \text{ and } v \text{ lies on the boundary of } f \text{ in } \Gamma\}.\]

If $\Gamma$ is also $2$-connected, then $R(\Gamma)$ is a quadrangulation of the sphere i.e., a $2$-connected, plane graph where every region is delimited by a $4$-cycle \cite[section 2.8]{mohtho}. A pair of vertices is adjacent in $\Gamma$ if and only if the corresponding ones in $R(\Gamma)$ are (opposite vertices) on the same quadrangular region.

Further, if $\Gamma$ is $3$-connected then so is its radial graph \cite[Lemma 2.1]{arcric}. If $\Gamma$ is a $2$-connected plane graph, then $\Gamma$ is $3$-connected if and only if $R(\Gamma)$ has no separating $4$-cycles \cite[Lemma 2.8.2]{mohtho}.

Related to the radial graph, we mention here for the interested reader a construction of Shabat-Voevodsky \cite{shabat2013drawing} (see also notes by Grothendieck \cite{grothendieck1984esquisse}). In this construction, a $3$-coloured triangulation is associated to a planar map, and one may recover the radial graph of the planar map from two colour classes of the triangulation.

\subsection{The subgraphs $H_3$ and $H_+$}
\label{sec:H}

One strategy to prove that two graphs are non-isomorphic is to inspect certain respective subgraphs, for instance the ones defined below, and prove that they are not isomorphic.
\begin{defin}
	For any graph $\Gamma$, call respectively \[H_3(\Gamma) \qquad \text{ and } \qquad H_+(\Gamma)\]
	the subgraph generated by the vertices of degree $3$, and	the one generated by the vertices of degree $\geq 4$.
\end{defin}
For instance, $H_3(R(S(6,6)))\simeq\calP_8\cup\calP_8$, and $H_+(R(S(6,6)))\simeq\calP_2\cup\calP_2$.

\subsection{$\calZ$ transformation}
\label{sec:Z}
We consider the transformation $\calZ$ depicted in Figure \ref{fig:1a} (called $\calP$ in the notation of \cite{mafpo3}). It is applied to $R(\Gamma)$, where $\Gamma$ is a plane graph.
\begin{figure}[h!]
	\centering
	\begin{subfigure}[m]{0.4\textwidth}
		\centering
		\includegraphics[width=3.0cm,clip=false]{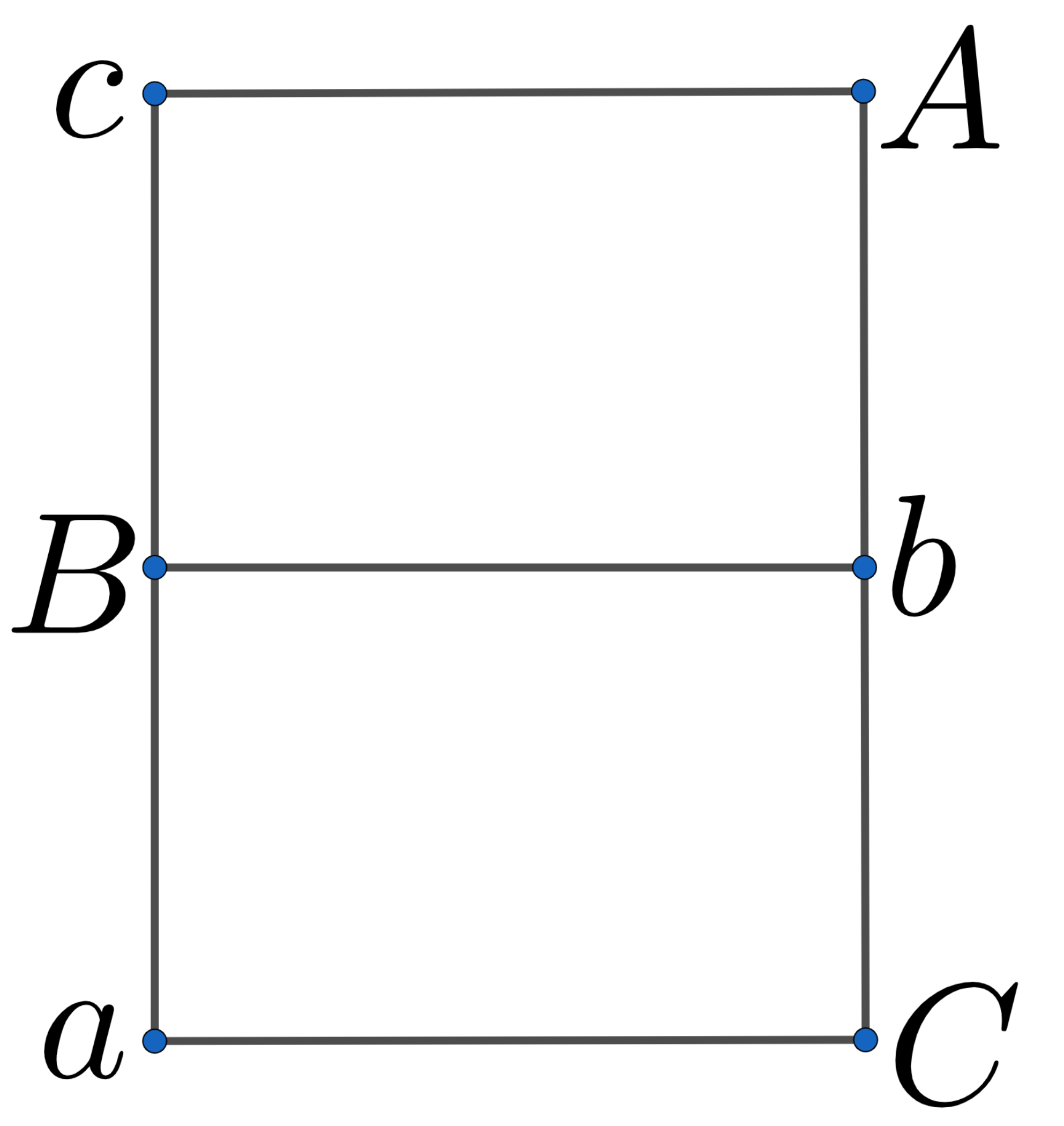}
	\end{subfigure}
	\hspace{-0.75cm}
	$\xrightarrow{\displaystyle\calZ}$
	\hspace{-0.75cm}
	\begin{subfigure}[m]{0.4\textwidth}
		\centering
		\includegraphics[width=3.0cm,clip=false]{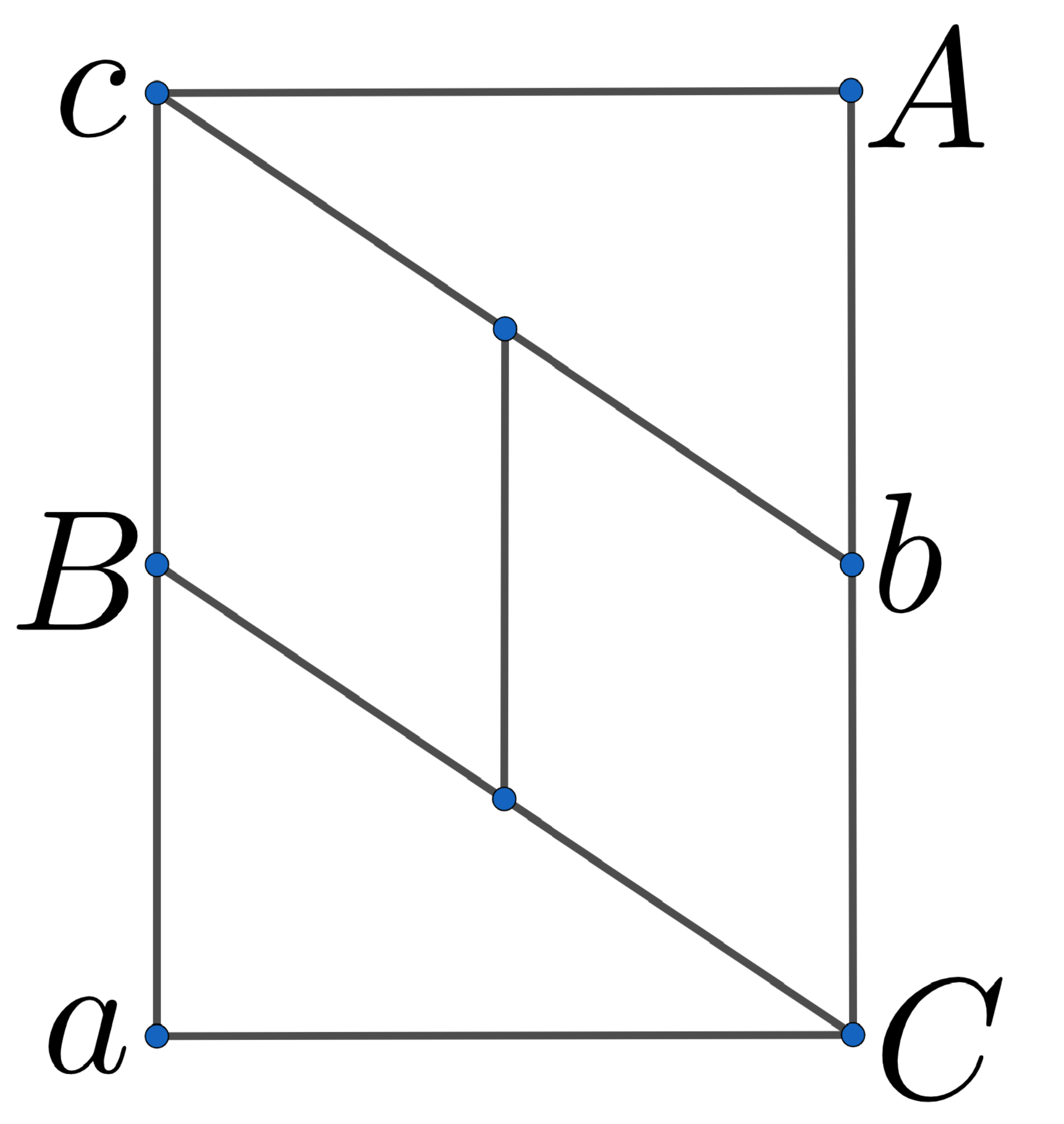}
	\end{subfigure}
	\caption{The transformation $\calZ$ applied to the radial graph of a graph $\Gamma$. Vertices of $\Gamma$ have lower-case labels, those of $\Gamma^*$ upper-case. The degree of $c,C$ increases by $1$, and two new vertices are inserted.}
	\label{fig:1a}
\end{figure}

It was proven in \cite[section 3.2]{mafpo3} that applying $\calZ$ to $R(\Gamma)$ preserves planarity and $3$-connectivity of the original graph. Further, the operation $\calZ$ acts on the radial graph in such a way that $\Gamma$ and $\Gamma^*$ are transformed in exactly the same fashion, preserving self-duality \cite[Remark 3.6]{mafpo3}. Note that in general, applying $\calZ$ with a different labelling on the radial graph does not necessarily preserve self-duality of $\Gamma$.

More in detail, the operation $\calZ$ corresponds to a so-called `edge-splitting' simultaneously in both $\Gamma$ and $\Gamma^*$. To split an edge $ab\in E(\Gamma)$, one performs
\begin{equation}
	\label{eqn:split}
	\Gamma-ab+d+da+db+dc,
\end{equation}
where $c$ is a vertex on the contour of a face containing $ab$ in $\Gamma$, $d$ is a new vertex, and $-/+$ stand for vertex or edge deletion/addition. In $\Gamma$, the number of vertices and faces increase by one, the new vertex $d$ is of degree $3$, the vertex $c$ increases its degree by one (it is now adjacent to $d$), and the degrees of all other vertices are unchanged. Planarity is preserved.

More generally, in rigidity theory, a $1$-extension of a $3$-connected generic circuit is another $3$-connected generic circuit. In fact, every $3$-connected generic circuit may be obtained by applying $1$-extensions on an initial tetrahedron \cite{berjor}. For rigidity theory in general, see e.g.\ \cite{lova82}.


\subsection{Algorithm to construct self-dual $3$-polytopes of given sequence}
\label{sec:pre}

\begin{alg}[{\cite[Algorithm 3.2]{mafpo3}}]
		\label{alg:1}
		\noindent\textbf{Input.} A $k$-tuple of integers \[T=(t_1,t_2,\dots,t_k),\] with $t_i\geq 4$ for each $i$.
		\\
		\noindent\textbf{Output.} A self-dual polyhedron $P(T)$ of degree sequence \eqref{eqn:seq}.
		\\
		\noindent\textbf{Description.} 
One starts with the cube, radial graph of the tetrahedron (the smallest polyhedron and self-dual). First, one applies the transformation $\calZ$ in Figure \ref{fig:1a} (to any two adjacent faces of the cube). Next, one performs a relabelling, in either of two ways -- Figures \ref{fig:1b} and \ref{fig:1c}. The choice between these two depends on the first entry $t_1$ of the inputted tuple
\[T=(t_1,\dots,t_k).\]
If $t_1\geq 5$, we relabel as in Figure \ref{fig:1b} before reapplying $\calZ$, whereas if $t_1=4$, we relabel as in Figure \ref{fig:1c}. One also decreases $t_1$ by $1$ before proceeding, eliminating it from the tuple when it drops below the value $4$. The procedure stops when $T$ is empty. To summarise, for each $t_i$ in turn, $\calZ$ is performed a total of $t_i-4$ times. After each time except the last, we relabel as in Figure \ref{fig:1b}, whereas after the $t_i-4$-th application, we relabel as in Figure \ref{fig:1c} before proceeding to work on $t_{i+1}$ (and stopping after $t_k$). The result is $R(P(T))$. Finally, one recovers $P(T)$ from its radial graph.

\begin{figure}[h!]
	\centering
	\begin{subfigure}{0.44\textwidth}
		\centering
		\includegraphics[width=3.0cm,clip=false]{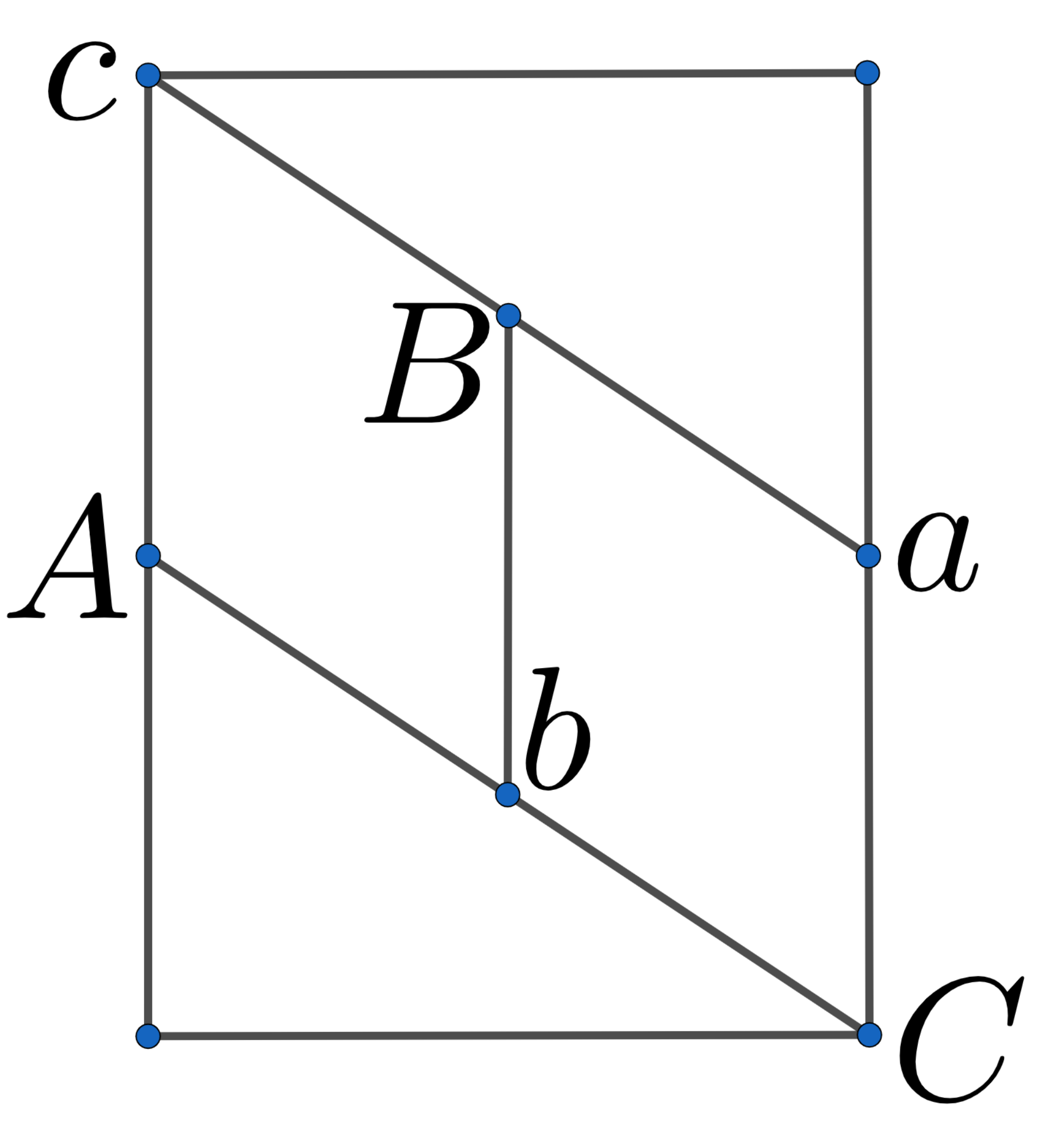}
		\caption{Relabelling the vertices when $t_i\geq 5$. Those with labels $c,C$ will continue to increase their degrees upon applying $\calZ$ again.}
		\label{fig:1b}
	\end{subfigure}
	\hspace{1.0cm}
	\begin{subfigure}{0.44\textwidth}
		\centering
		\includegraphics[width=3.0cm,clip=false]{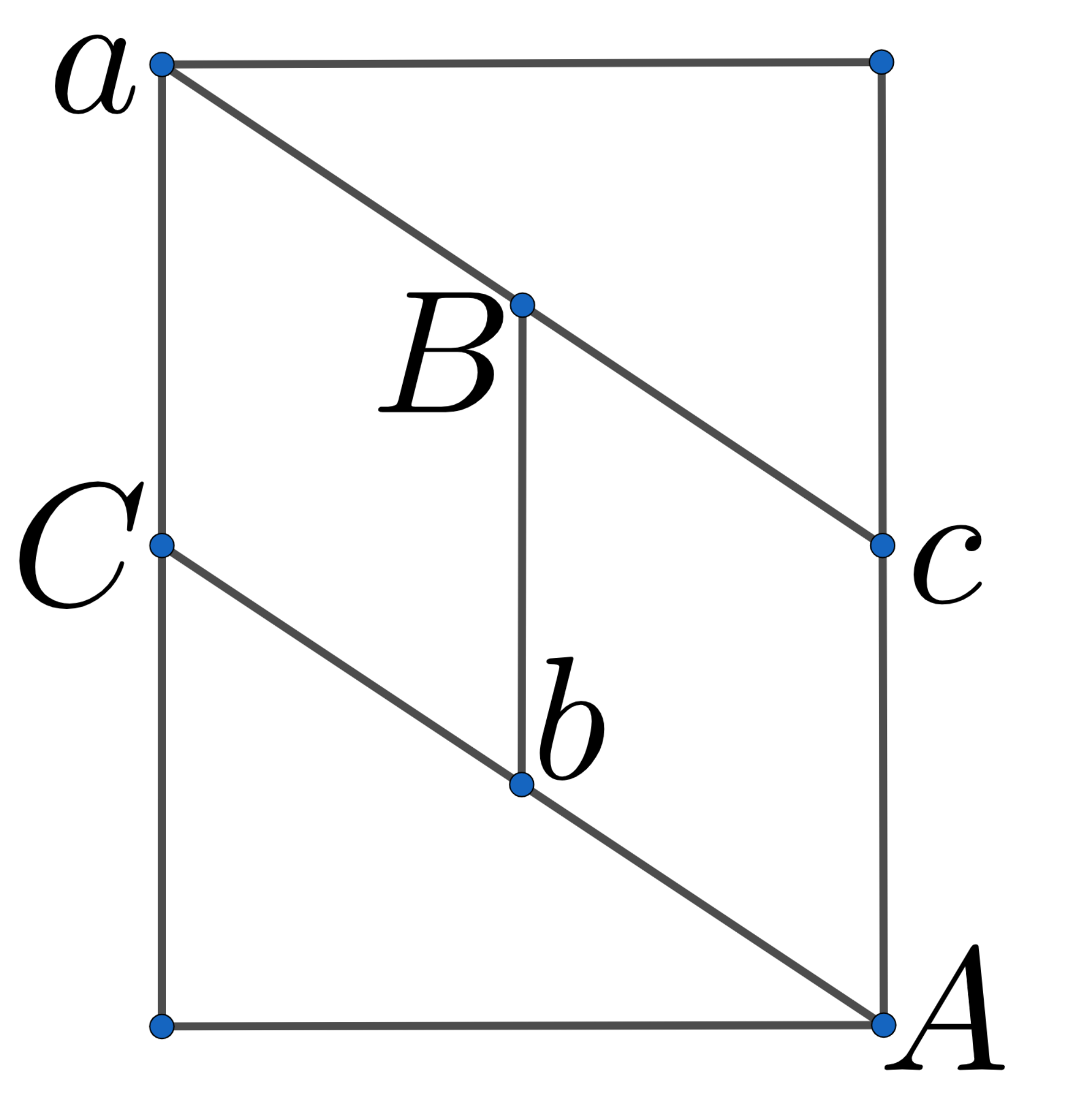}
		\caption{Relabelling the vertices when $t_i=4$. The degrees of $a,A$ will not change any more, while those of $c,C$ will increase by $1$ with the next application of $\calZ$.}
		\label{fig:1c}
	\end{subfigure}
	\caption{Relabelling the vertices after the transformation $\calZ$.}
	\label{fig:1bc}
\end{figure}
\end{alg}
The example $R(P((6,6)))$ is depicted in Figure \ref{fig:66}.
\begin{figure}[h!]
	\centering
	\includegraphics[width=7.0cm,clip=false]{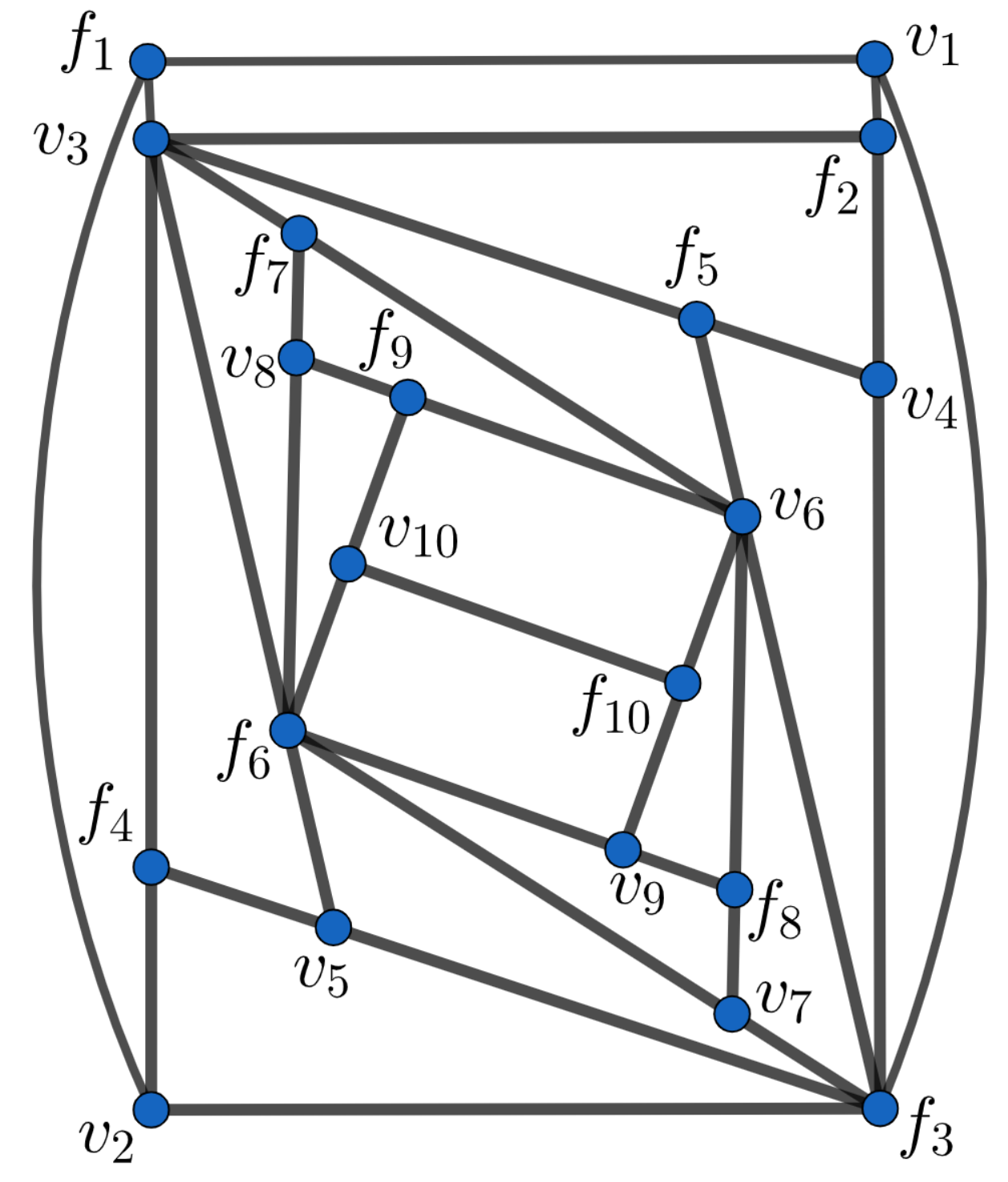}
	\caption{The graph $R(P((6,6)))$, radial graph of the self-dual $3$-polytope $P((6,6))\simeq S(6,6)$. The bijection $\varphi: v_i\mapsto f_i$ determines a graph isomorphism between $P((6,6))$ and its dual. To construct $P((6,6))$ via Algorithm \ref{alg:1}, one starts with an initial cube of vertices $\{v_1,f_1,v_2,f_2,v_3,f_3,v_4,f_4\}$. Then one assigns $a:=v_2$, $A:=f_2$, $c:=v_3$, $C:=f_3$, $b:=v_4$, $B:=f_4$ as in Figure \ref{fig:1a}, to implement the algorithm with input $(6,6)$. The first pair of vertices to be inserted is $v_5,f_5$, then $v_6,f_6$, and so forth up to $v_{10},f_{10}$. Lastly one recovers $P((6,6))$ from $R(P((6,6)))$.}
	\label{fig:66}
\end{figure}

\begin{rem}
If the above procedure is stopped short after $t_i$ applications of $\calZ$ for each $1\leq i\leq k'-1< k$, and $1\leq j\leq t_{k'}$ further applications for $t_{k'}$, then the output is (the radial graph of)
\[P((t_1,t_2,\dots,t_{k'-1},j)).\]
\end{rem}

\subsection{Strengths of Algorithm \ref{alg:1}}
\label{sec:deep}
The power behind this procedure was showcased in \cite{mafpo3}, where we proved that it produces a self-dual $3$-polytope for any choice of admissible degree sequence \eqref{eqn:seq}.

As stated in the introduction, another important feature of Algorithm \ref{alg:1} is that it allows to increase one vertex degree at a time, leaving the rest unchanged, and inserting new vertices only of degree $3$. This feature makes the algorithm very versatile. It will be handy for instance when defining modifications of the algorithm later on.

We define a `canonical' labelling for the vertices of $R(P(T))$,
\begin{equation}
\label{eqn:can}
\{v_1,f_1,v_2,f_2,\dots,v_{k+m},f_{k+m}\},
\end{equation}
where the $v_i$'s are vertices of $P(T)$, and the $f_i$'s of its dual. Labels are assigned as follows. For the initial cube we have \[\{v_1,f_1,v_2,f_2,v_3,f_3,v_4,f_4\}.\]
To perform the algorithm, we set $a:=v_2$, $A:=f_2$, $c:=v_3$, $C:=f_3$, $b:=v_4$, $B:=f_4$ as in Figure \ref{fig:1a}. Each application of $\calZ$ introduces two new vertices, that we label in turn $v_5,f_5$, then $v_6,f_6$, and so forth up to $v_{k+m},f_{k+m}$. With such labelling, the map
\begin{align}
	\label{eqn:phi}
	\notag\varphi: V(P(T))&\to V((P(T))^*)\\
	v_i&\mapsto f_i
\end{align}
defines a graph isomorphism between $P(T)$ and its dual (this is tantamount to proving that, in the algorithm, each application of $\calZ$ preserves self-duality \cite{mafpo3}). The case $T=(6,6)$ is sketched in Figure \ref{fig:66}.

\section{Only $x,y,3^{x-y-4}$ is unigraphic in the class of self-dual polyhedra: proof of Theorem \ref{thm:1}}
\label{sec:proof}

\subsection{Outline of proof for Theorem \ref{thm:1}}
\label{sec:out}
In section \ref{sec:1p}, we will prove via direct constructions that each $S(x,y)$ of Definition \ref{def:S} is unigraphic in the class of self-dual polyhedra. We will end the section with the proof of Corollary \ref{cor}.

As remarked in \cite[section 3]{mafpo3}, on inputting permutations of the tuple $T=(t_1,\dots,t_k)$ into Algorithm \ref{alg:1}, one may obtain non-isomorphic graphs. This is one of the ideas that we will use in section \ref{sec:2p} to show that `most' of the sequences \eqref{eqn:seq} have more than one non-isomorphic self-dual realisation. To show that they are not isomorphic, we will inspect the respective subgraphs $H_+(G)$.

A type of sequence left out in the above reasoning is
\begin{equation}
	\label{eqn:seqnk}
	n^k,3^{4+(n-4)k}, \quad n\geq 4, k\geq 3,
\end{equation}
where the entries of the tuple are all equal, giving rise to only one permutation. To prove non-unigraphicity for \eqref{eqn:seqnk}, in section \ref{sec:3p} we will apply Algorithm \ref{alg:1} with a starting graph different from the cube. Comparing the subgraphs $H_+(R(G))$, we will show that the new family of graphs are not isomorphic to the corresponding
\[P((\underbrace{n,n,\dots,n}_{k \text{ times}})),\]
obtained via Algorithm \ref{alg:1} on the cube. This will work for $n\geq 5$.



Lastly, for the case
\[4^k,3^4, \qquad k\geq 3\]
in section \ref{sec:4p} we will write Algorithm \ref{alg:2} (see below) to produce a family of self-dual polyhedra, not isomorphic to the ones constructed via Algorithm \ref{alg:1}. To show that they are not isomorphic, we will analyse the respective subgraphs $H_3(G)$.

\subsection{Unigraphicity of $x,y,3^{x-y-4}$ in the class of self-dual polyhedra}
\label{sec:1p}
For $y=3$ the result is clear. Let us construct a self-dual $3$-polytope $S(x,y)$ of sequence
\[x,y,3^{x-y-4}, \quad x\geq y\geq 4.\]
By self-duality, there are $x-y-2$ faces, of which one $x$-gonal, one $y$-gonal, and the rest triangular. In any polyhedron, two distinct faces share either $0,1,$ or $2$ vertices (in the latter case, they share an edge, i.e., they are adjacent faces). The $x$-gon and $y$-gon must be adjacent -- c.f. Figure \ref{fig:75}, otherwise $S(x,y)$ would have
\[\geq x+y-1\]
vertices, contradiction. We now assign the labelling
\begin{equation}
	\label{eqn:XY}
X=[v_1,v_2,\dots,v_x], \qquad Y=[w_1,w_2,\dots,w_{y-2},v_{x},v_{x-1}]
\end{equation}
to the vertices of the $x$-gon and $y$-gon (as in Figure \ref{fig:75}). We are using the somewhat standard square bracket notation for a face of a $3$-polytope. Of the $x-y-2$ vertices
\begin{equation}
	\label{eqn:XY2}
V(S(x,y))=\{v_1,\dots,v_x,w_1,\dots,w_{y-2}\},
\end{equation}
one has degree $x$, one $y$, and the remaining $3$. All faces other than $X,Y$ are triangular. Since no diagonal of $X,Y$ may be an edge, we have
\[\deg(v_{x-1})=\deg(v_x)=3.\]
For the same reason, a vertex may have degree $x$ only if it is on the boundary of $Y$, and it is adjacent to all of $v_1,\dots,v_{x-2}$ and to its two neighbours on $Y$. Since there are only two vertices of degree greater than $3$, another vertex may have degree $y$ only if it is on the boundary of $X$, and it is adjacent to all of $w_1,\dots,w_{y-2}$ and to its two neighbours on $X$. That is to say,
\[\deg(v_{i})=y \qquad \text{and} \qquad \deg(w_j)=x,\]
with $1\leq i\leq x-2$ and $1\leq j\leq y-2$ to be determined. We claim that either $i=1$ and $j=1$, or $i=x-2$ and $j=y-2$. Indeed, otherwise, the partition
\[\{\{v_{x-1},v_i,w_{j}\},\{v_{x},v_{i+1},w_{j-1}\}\}\]
would determine a $K(3,3)$-minor in $S(x,y)$, contradicting planarity (Kuratowski's Theorem). Up to relabelling, we have $i=j=1$. The self-dual polyhedron is uniquely determined, hence there is a unique self-dual realisation for \eqref{eqn:sxy}, as claimed.

We have just completed the first part of the proof of Theorem \ref{thm:1}. We are also ready to prove Corollary \ref{cor}.
\begin{proof}[Proof of Corollary \ref{cor}]
Thanks to Theorem \ref{thm:1}, we only need to study sequences of type \eqref{eqn:sxy}. Each has the unique self-dual realisation $S(x,y)$. Now consider $S(x,y-1)$, $x\geq y\geq 4$, labelled as in \eqref{eqn:XY2} and \eqref{eqn:XY} -- cf. Figure \ref{fig:75}. We define \[Q(x,y):=S(x,y-1)-v_3v_4+u+uv_1+uv_3+uv_4, \qquad x\geq y\geq 4,\]
where $u$ is a new vertex. This operation is an edge-splitting, hence $Q(x,y)$ is another $3$-polytope. The degree sequence of $S(x,y-1)$ is
\[x,y-1,3^{x+(y-1)-4}.\]
with $\deg(v_1)=y-1$. Therefore, the sequence of $Q(x,y)$ is \eqref{eqn:sxy},
\[x,y,3^{x+y-4}.\]
On the other hand, the faces of this new polyhedron are the $x-1$-gon
\[[v_1,u,v_4,v_5,\dots,v_x],\]
the $4$-gons
\[[v_1,v_2,v_3,u] \quad \text{ and } \quad [v_3,u,v_4,w_1],\]
a $y-1$-gon, and the remaining are triangles. In particular, as soon as $x\geq 5$, $Q(x,y)$ is not self-dual, thus $Q(x,y)\not\simeq S(x,y)$. Therefore, \eqref{eqn:sxy} has at least two polyhedral realisations when $x\geq 5$ and $y\geq 4$. On the other hand, pyramids and $S(4,4)\simeq Q(4,4)$ are unigraphic polyhedra.
\end{proof}

\subsection{Non-unigraphicity of $t_1,\dots,t_k,3^m$ when $k\geq 3$ and the $t_i$'s are not all equal}
\label{sec:2p}
We want to show that the degree sequence
\begin{align}
	\label{eqn:s2}
	&\notag t_1,\dots,t_k,3^m, \qquad k\geq 3, 
	\\&\text{ where the } t_i \text{'s are not all equal},
\end{align}
and $m$ is given by \eqref{eqn:m} as usual, has at least two non-isomorphic self-dual realisations.

At the end of this section, we will prove the following. It will require a careful analysis of Algorithm \ref{alg:1}.
\begin{lemma}
	\label{lem:leaf}
Let $G$ be the output of Algorithm \ref{alg:1} when we input the tuple $(t_1,\dots,t_k)$, and $u_1,\dots,u_k$ the vertices of $G$ satisfying
\[\deg(u_i)=t_i, \quad 1\leq i\leq k.\]
Then for each $1\leq i\leq k-1$, the vertex $u_{i}$ is adjacent in $G$ to $u_{i+1}$. Moreover, $u_{i}$ is adjacent to at most one more $u_j$, $j\geq i+2$. If $1\leq i\leq k-2$ and $t_{i+1}=5$, then $u_{i}$ is adjacent to $u_{i+2}$. If $1\leq i\leq k-3$ and $t_{i+1}=t_{i+2}=4$, then $u_{i}$ is adjacent to $u_{i+3}$. In all other cases, $u_i$ is not adjacent to any $u_j$, $j\geq i+2$.
\end{lemma}

Applying Lemma \ref{lem:leaf}, we will prove that \eqref{eqn:s2} has at least two non-isomorphic self-dual realisations. With no loss of generality, we can assume that
\begin{equation}
\label{eqn:cond}
t_1\neq t_2, \quad t_1\neq 5.
\end{equation}
Call $G_1,G_2$ the graphs constructed by Algorithm \ref{alg:1} on inputting
\[(t_1,t_2,t_3,\dots,t_k)
\quad \text{ and } \quad
(t_2,t_1,t_3,\dots,t_k)\]
respectively. We will now apply Lemma \ref{lem:leaf} to show that $H_+(G_1)\not\simeq H_+(G_2)$ by comparing their end-vertices (i.e., vertices of degree $1$), so that $G_1\not\simeq G_2$, and therefore \eqref{eqn:s2} has at least two self-dual realisations. 

If $k=3$ and $t_2=5$, then $H_+(G_1)$ has no end-vertices, while $H_+(G_2)$ has two (since $t_1\neq 5$). If $k=3$ and $t_2\neq 5$, then the end-vertices of $H_+(G_1)$ are of degrees $t_1$ and $t_3$ in $G_1$, while the end-vertices of $H_+(G_2)$ are of degrees $t_2$ and $t_3$ in $G_2$. In either case, $G_1\not\simeq G_2$.

Now let $k\geq 4$. Unless the value $4$ appears $k-1$ times in the tuple, we can take
\[\quad t_3\neq 4\]
in addition to the conditions \eqref{eqn:cond}. As in Lemma \ref{lem:leaf}, $u_i$ is the label for the vertex of degree $t_i$, $1\leq i\leq k$. Since $t_3\neq 4$, we record that $u_k$ is an end-vertex in $H_+(G_1)$ if and only if it is an end-vertex in $H_+(G_2)$ (still due to Lemma \ref{lem:leaf}). Hence in $H_+(G_2)$, the end-vertices are $u_2$ and possibly $u_k$. On the other hand, in $H_+(G_1)$, only $u_1$ and $u_k$ may be end-vertices. Since $t_1\neq t_2$, we conclude that in this case $G_1\not\simeq G_2$.

Lastly, we take the two $k$-tuples
\[(4,t_2,4,\dots,4)
\quad \text{ and } \quad
(t_2,4,4,\dots,4),\]
with $k\geq 4$. 
If $k=4$, then $H_+(G_1)$ has at least one end-vertex, while $H_+(G_2)$ has none. Now let $k\geq 5$. If $t_2\geq 6$, then $H_+(G_1)$ has one end-vertex, while $H_+(G_2)$ has none. If $k\geq 5$ and $t_2=5$, then $u_2$ i.e., the vertex of degree $5$ in $G_1$ and $G_2$, has respectively degree $3$ and $2$ in $H_+(G_1)$ and $H_+(G_2)$.

We have seen that, in any case, $G_1\not\simeq G_2$. We end this section by proving Lemma \ref{lem:leaf}.

\begin{proof}[Proof of Lemma \ref{lem:leaf}]\
In this proof, we shall always assume that $1\leq i\leq k-1$. 
	Referring to Figure \ref{fig:1c}, after raising the degree of $u_i$ from $t_i-1$ to $t_i$ via the algorithm, we have $a=u_i$ and $c=u_{i+1}$. The remaining graph transformations will not modify the face of the radial graph containing $a,B,c$, hence $u_i,u_{i+1}$ are adjacent in the outputted $3$-polytope.
	
	More generally, the degree of $u_i$ is raised from $3$ to $t_i$ via $t_i-3$ successive applications of $\calZ$. We refer to these, in order, as
	\[\calZ_\ell, \quad 1\leq\ell\leq t_i-3.\]
	Each such application adds a pair of vertices to the graph. In the output $G$, $u_i$ is adjacent:
	\begin{itemize}
	\item
	to three vertices that already belonged to the graph before performing $\calZ_1$;
	\item
	for $1\leq\ell\leq t_i-4$, to one of the two vertices added during the application of $\calZ_\ell$;
	\item
	to one of the two vertices that are added during the first application of $\calZ$ for $t_{i+1}$. We will refer to this vertex as $x_i$.
	\end{itemize}
	
	If $t_i\geq 5$, then $u_{i+1}$ is added to the graph while performing $\calZ_{t_i-4}$. If instead $t_i=4$, then $u_{i+1}$ already belonged to the graph before the application of $\calZ_{1}$.
	
	As for $x_i$, if $1\leq i\leq k-2$ and $t_{i+1}=5$, then $x_{i}=u_{i+2}$. This situation is illustrated in Figure \ref{fig:l5} (cf. Figure \ref{fig:1bc}). If $1\leq i\leq k-3$ and $t_{i+1}=t_{i+2}=4$, then we have $x_i=u_{i+3}$, as in Figure \ref{fig:l4}. In all other cases, $x_i$ has degree $3$ in $G$. Two examples are shown in Figures \ref{fig:l6} and \ref{fig:l7}. In each case, the remaining graph transformations will not modify the face of the radial graph containing $u_{i},x_{i}$, hence these are adjacent in the output.

\begin{figure}[h!]
	\centering
	\begin{subfigure}{0.23\textwidth}
		\centering
		\includegraphics[width=3.5cm,clip=false]{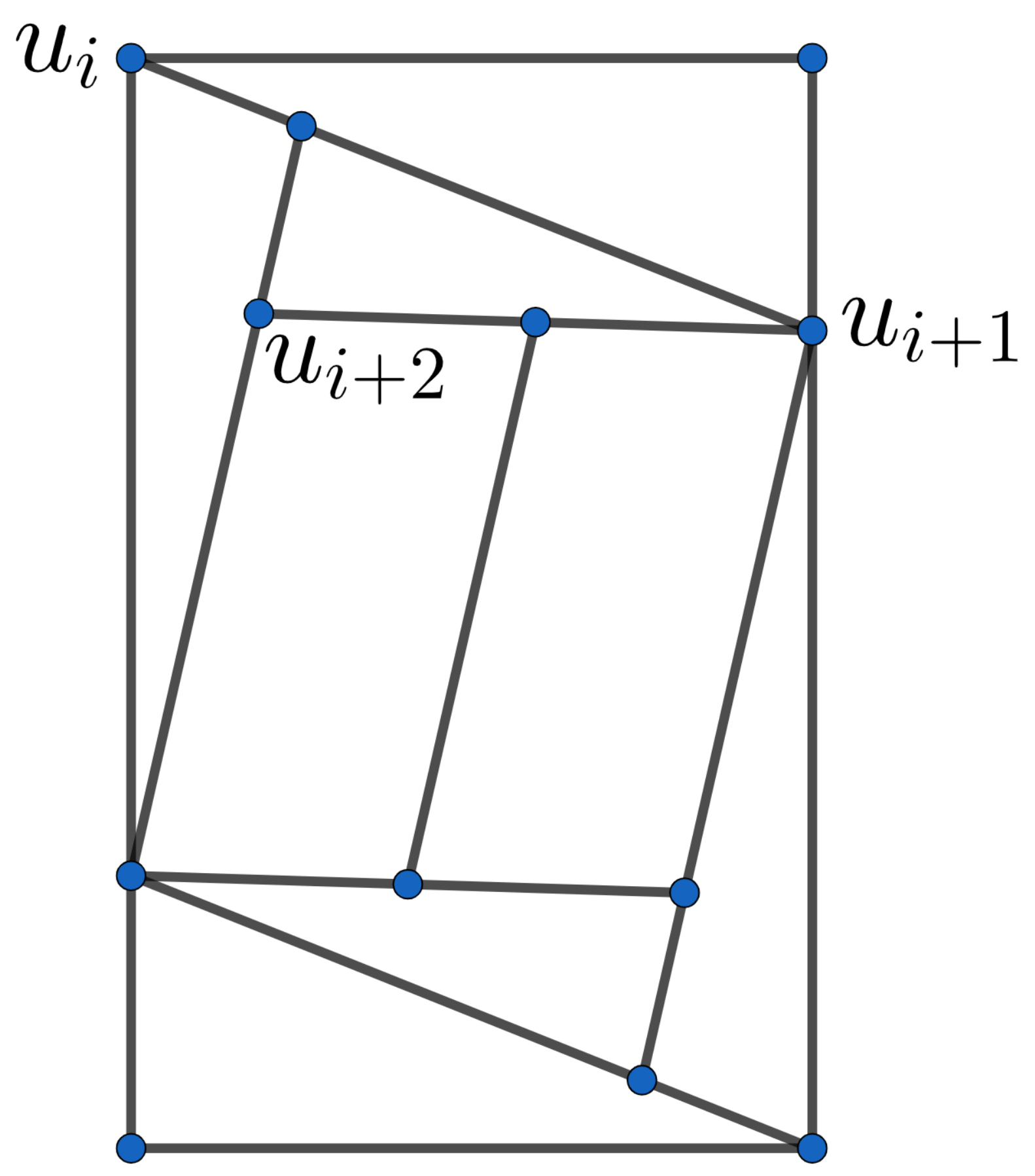}
		\caption{$t_{i+1}=5$.}
		\label{fig:l5}
	\end{subfigure}
	\begin{subfigure}{0.23\textwidth}
		\centering
		\includegraphics[width=3.5cm,clip=false]{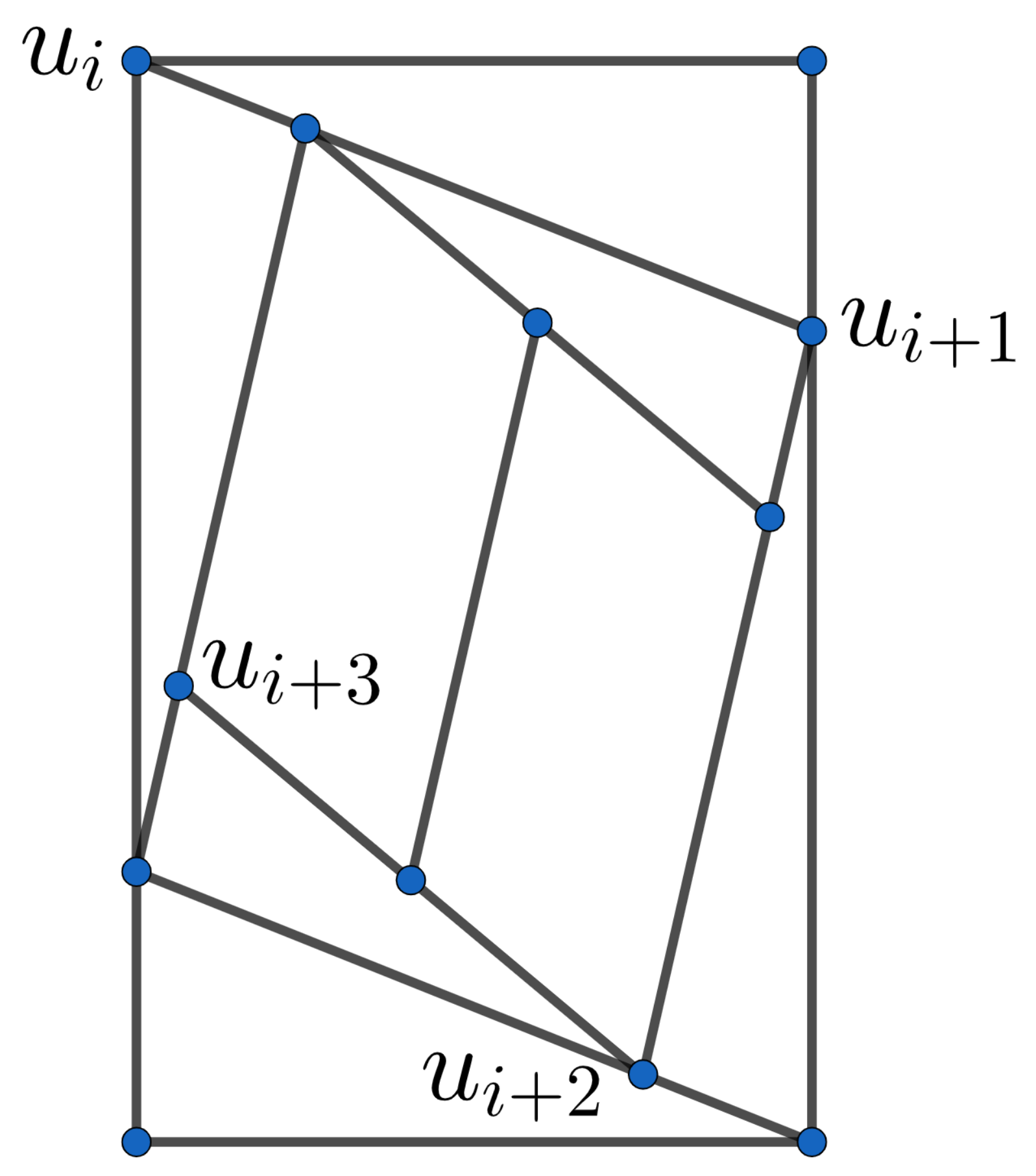}
		\caption{$t_{i+1}=t_{i+2}=4$.}
		\label{fig:l4}
	\end{subfigure}
	\begin{subfigure}{0.23\textwidth}
		\centering
		\includegraphics[width=3.5cm,clip=false]{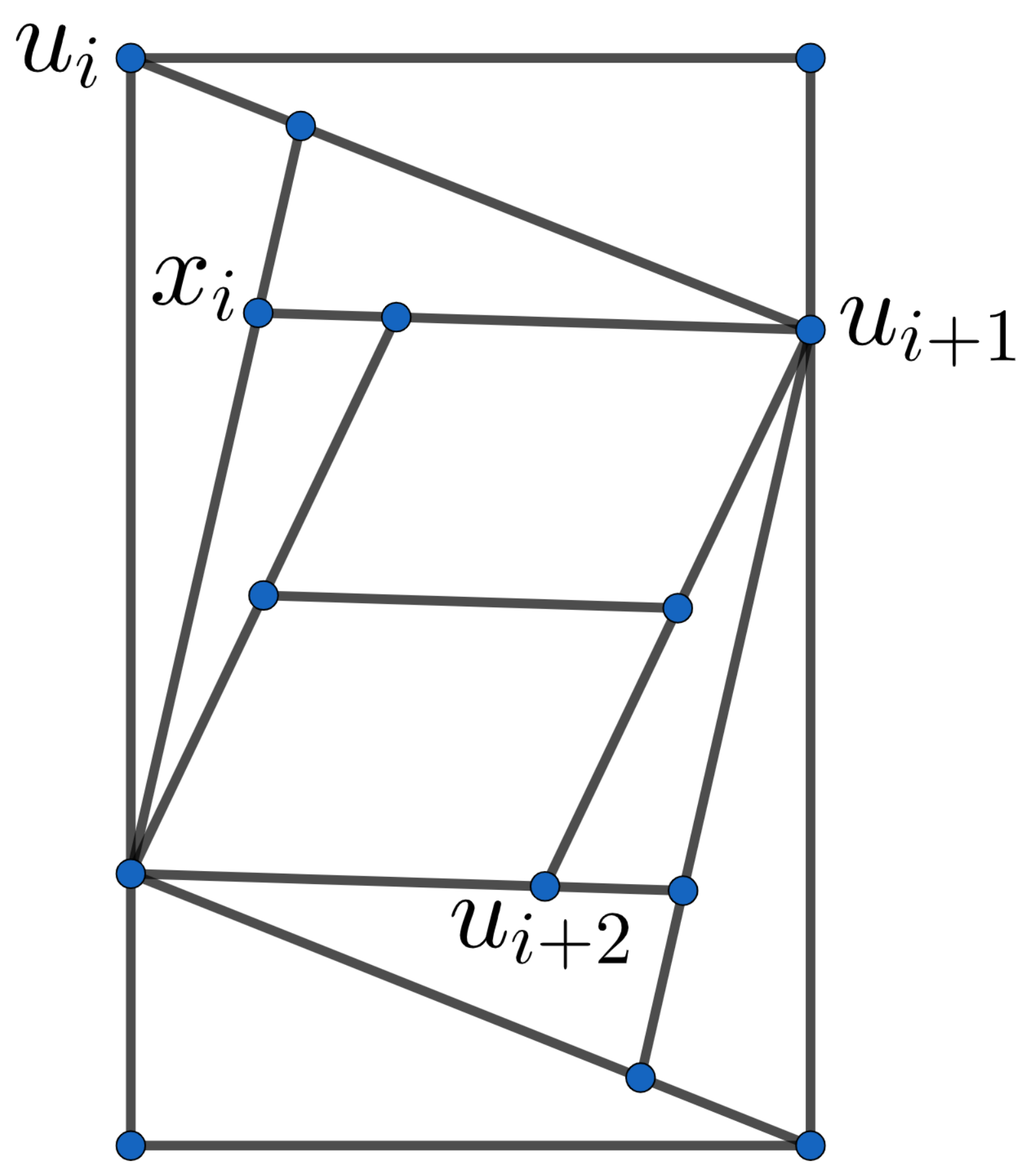}
		\caption{$t_{i+1}=6$.}
		\label{fig:l6}
	\end{subfigure}
	\begin{subfigure}{0.23\textwidth}
		\centering
		\includegraphics[width=3.5cm,clip=false]{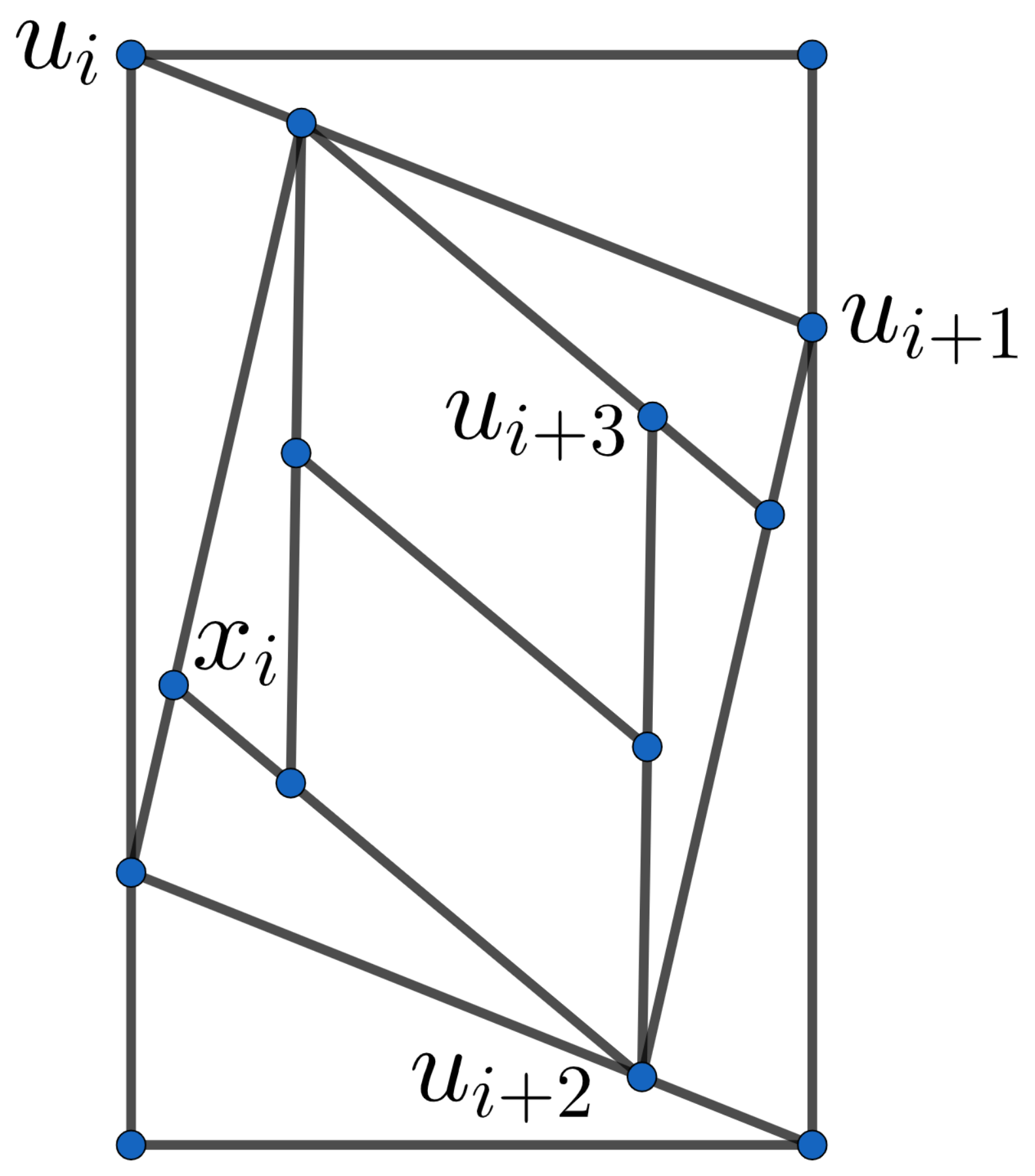}
		\caption{$t_{i+1}=4$, \ $t_{i+2}=5$.}
		\label{fig:l7}
	\end{subfigure}
	\caption{Various situations that can arise for the vertex $x_i$, depending on $t_{i+1},t_{i+2}$.}
	\label{fig:l}
\end{figure}

\end{proof}

\subsection{Non-unigraphicity of $n^k,3^{4+k(n-4)}$ when $n\geq 5$ and $k\geq 3$}
\label{sec:3p}
Here we will see that the degree sequence
\begin{equation}
	\label{eqn:seqa}
n^k,3^{4+k(n-4)}, \qquad n\geq 5, \ k\geq 3
\end{equation}
has at least two non-isomorphic self-dual realisations. One of these is
\[P((\underbrace{n,n,\dots,n}_{k \text{ times}})),\]
obtained via Algorithm \ref{alg:1}. For its implementation, we start with the cube, radial graph of the tetrahedron, and input the constant $k$-tuple $(n,n,\dots,n)$. As new vertices are inserted into the graph, the length of the tuple decreases. When this length has reached $k-2$, the resulting graph is the radial graph of $S(n,n)$ from Definition \ref{def:S}. 
Recall the canonical labelling \eqref{eqn:can}, where we assign increasing indices to newly inserted vertices in turn. Here $k=2$ and $t_1=t_2=n$, hence \[k+m=k+4+(t_1-4)+(t_2-4)=2n-2.\]
The canonical labelling for $R(S(n,n))$ is therefore
\begin{equation*}
	\{v_1,f_1,v_2,f_2,\dots,v_{2n-2},f_{2n-2}\}.
\end{equation*}
The map \eqref{eqn:phi}
\begin{align*}
\notag\varphi: V(S(n,n))&\to V((S(n,n))^*)\\
v_i&\mapsto f_i
\end{align*}
is a graph isomorphism.

We now define a new labelling for $S(n,n)$ as follows. The two $n$-gons are
\[f_1=[v_1,v_3,\dots,v_{2n-3},v_2]\]
and
\[f_{2n-2}=[v_2,v_4,\dots,v_{2n-2},v_{2n-3}].\]
The vertex $v_1$ is adjacent to $v_3$ and to $v_2,v_4,\dots,v_{2n-2}$ in this order around $v_1$ in the planar immersion, and the vertex $v_{2n-2}$ is adjacent to $v_{2n-4}$ and to $v_1,v_3,\dots,v_{2n-3}$ in this order around $v_{2n-2}$ in the planar immersion. The face $f_1$ is adjacent to $f_3$ and $f_2,f_4,\dots,f_{2n-2}$ in this order around $f_1$ in the planar immersion, and the face $f_{2n-2}$ is adjacent to $f_{2n-4}$ and to $f_1,f_3,\dots,f_{2n-3}$ in this order around $f_{2n-2}$ in the planar immersion, as in Figure \ref{fig:green}. With this new labelling, \eqref{eqn:phi} is still an isomorphism.
\begin{figure}[h!]
	\centering
	\includegraphics[width=5.0cm,clip=false]{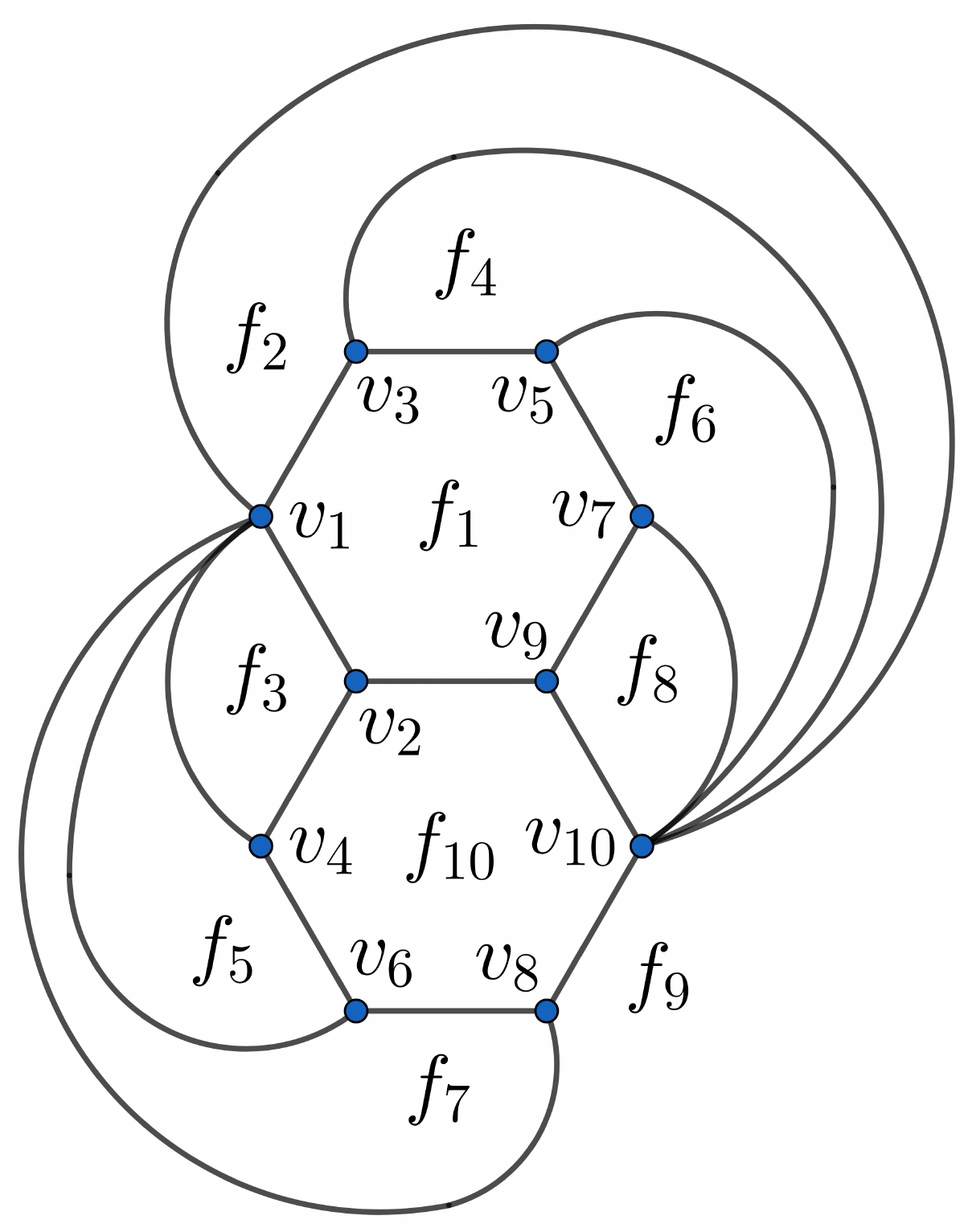}
	\caption{New labelling for $S(6,6)$.}
	\label{fig:green}
\end{figure}

The set of edges in the radial graph $R(S(n,n))$ is given by the cycle
\[v_1,f_2,v_3,\dots, v_{2n-3},f_{2n-2},v_{2n-2},f_{2n-3},\dots,v_2,f_1,v_1,\]
together with
\begin{align*}
	&v_1f_i, \qquad i \text{ odd, } \quad 3\leq i\leq 2n-3;
	\\&f_1v_i, \qquad i \text{ odd, } \quad 3\leq i\leq 2n-3;
	\\&v_{2n-2}f_i, \qquad i \text{ even, } \quad 2\leq i\leq 2n-4;
	\\&f_{2n-2}v_i, \qquad i \text{ even, } \quad 2\leq i\leq 2n-4.
\end{align*}
For the reader's convenience, we have also depicted $R(S(6,6))$ with the new labelling -- Figure \ref{fig:66post}, alongside the old one -- Figure \ref{fig:66pre}, copied from Figure \ref{fig:66}.
\begin{figure}[h!]
	\centering
	\begin{subfigure}{0.45\textwidth}
		\centering
		\includegraphics[width=5.5cm,clip=false]{S66.pdf}
		\caption{Canonical labelling, as in Figure \ref{fig:66}.}
		\label{fig:66pre}
	\end{subfigure}
	\hspace{0.50cm}
	\begin{subfigure}{0.45\textwidth}
		\centering
		\includegraphics[width=5.6cm,clip=false]{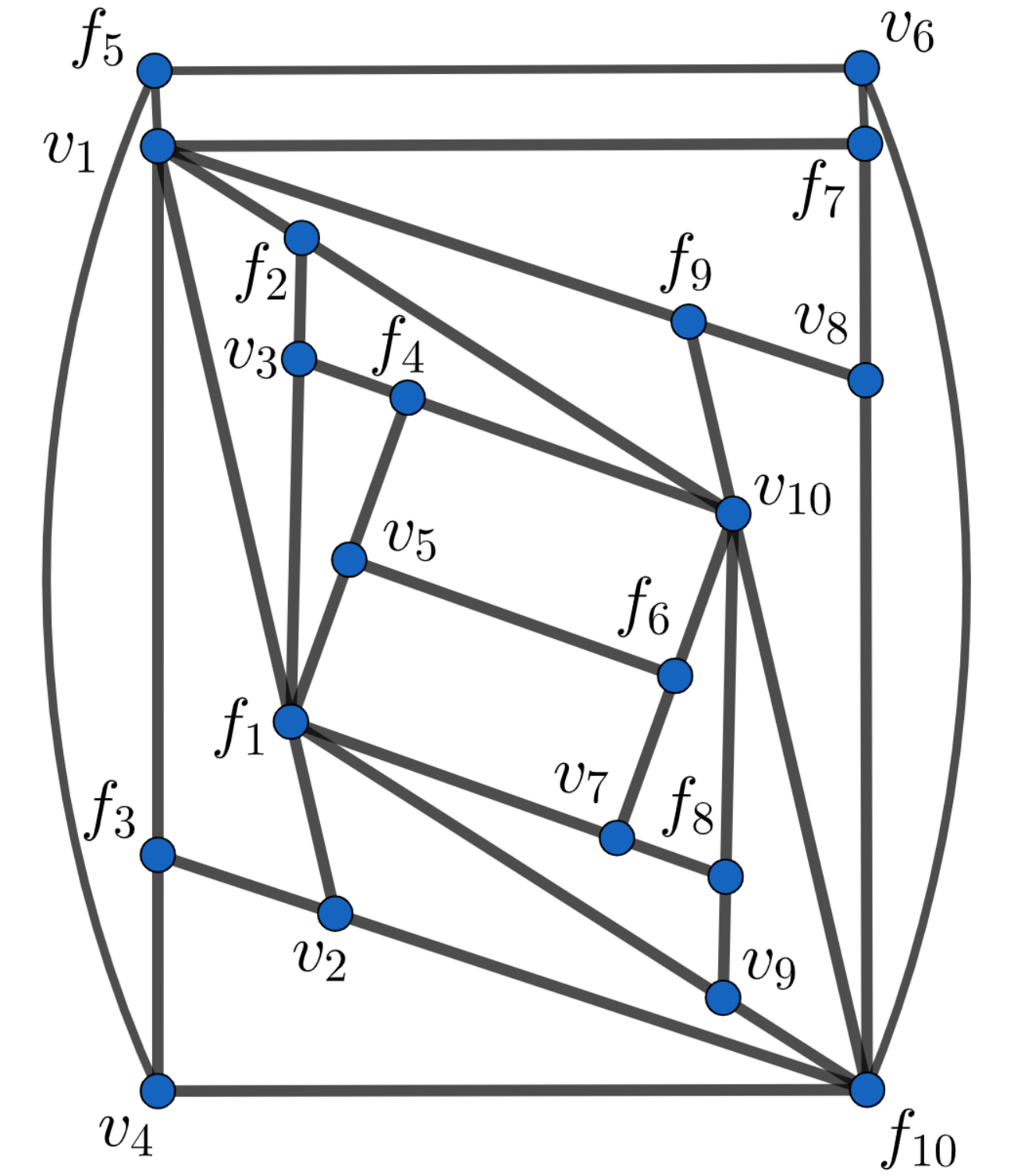}
		\caption{New labelling.}
		\label{fig:66post}
	\end{subfigure}
	\caption{The graph $R(S(6,6))$ with old and new labellings.}
	\label{fig:66new}
\end{figure}

We wish to apply Algorithm \ref{alg:1} starting with the graph $R(S(n,n))$ with its new labelling (instead of the cube). To this end, we assign $a:=v_2$, $A:=f_2$, $b:=v_1$, $B:=f_1$, $c:=v_3$, $C:=f_3$ referring to Figure \ref{fig:1bc}, and input a $k-2$-tuple $(n,n,\dots,n)$. The output is a certain graph $P'(n;k)$.

Our goal for this section is to show that $P'(n;k)$ is a self-dual realisation of \eqref{eqn:seqa}, and moreover
\begin{equation}
	\label{eqn:notiso1}
	P'(n;k)\not\simeq P((\underbrace{n,n,\dots,n}_{k\text{ times}})).
\end{equation}
As proven in \cite[section 3]{mafpo3}, if we apply Algorithm \ref{alg:1} to a self-dual $3$-polytope with a labelling as in Figure \ref{fig:green}, and assigning $a:=v_2$, $A:=f_2$, $b:=v_1$, $B:=f_1$, $c:=v_3$, $C:=f_3$, the output is another self-dual $3$-polytope. We point out that the condition $n\geq 5$ is essential for  $P'(n;k)$ to realise \eqref{eqn:seqa}. Indeed, in this way the degrees of $v_1,f_1$ in the output will be still $n$, as in the initial $S(n,n)$.

To show \eqref{eqn:notiso1}, we inspect the respective subgraphs $H_+(R(\cdot))$ of the LHS and RHS, where $R$ is the radial graph. Analysing the implementation of Algorithm \ref{alg:1}, we claim that in $R(P'(n;k))$, the vertices $v_{2n-2},f_{2n-2}$ have degree $n$, and moreover they are adjacent only to each other and to vertices of degree $3$ in the radial graph. Indeed, apart from $v_{2n-2},f_{2n-2}$ themselves, the vertices of $R(P'(n;k))$ of degree $n$ are $v_1,f_1,v_3,f_3$ (which are not adjacent to $v_{2n-2},f_{2n-2}$ in the radial graph), and possibly others, however these are separated from $v_{2n-2},f_{2n-2}$ by the $6$-cycle $v_1,f_2,v_3,f_1,v_2,f_3$. 
It follows that one of the connected components of
\[H_+(R(P'(n;k)))\]
is a copy of $K_2$.

On the other hand, in 
\begin{equation}
	\label{eqn:cumb}
H_+(R(P((\underbrace{n,n,\dots,n}_{k\text{ times}})))),
\end{equation}
the $2k$ vertices of degree $n$ are
\[v_{3+i},f_{3+i}, \qquad i=0,\ n-3,\ 2(n-3),\dots,\ (k-1)(n-3).\]
This happens because of how Algorithm \ref{alg:1} works: in Figure \ref{fig:1a}, the vertex labelled $c$ increases its degree from $3$ to $n$ via $n-3$ application of $\calZ$, while $2(n-3)$ new vertices are inserted. Hence if some $v_x$ has degree $n$ in the final graph, then the same will be true for $v_{x+(n-3)}$ (unless $v_{x+(n-3)}\geq p-1$, where $p=k+4+k(n-4)$ is the total number of vertices in \eqref{eqn:seqa}).

Further, for each $1\leq j\leq k-1$, the vertex $v_{3+(j-1)(n-3)}$ is adjacent to $f_{3+j(n-3)}$ (and similarly $f_{3+(j-1)(n-3)}$ to $v_{3+j(n-3)}$). This happens for the following reasons: with the change of labelling in Figure \ref{fig:1c}, the next vertices after $a,A$ to increase their degree with Algorithm \ref{alg:1} will be $c,C$; moreover, $a,C$ are adjacent, and will remain adjacent in the radial graph of the final outputted graph.

In particular, since $k\geq 3$, each vertex of degree $n$ in \eqref{eqn:cumb} lies on a path of length at least $3$ in \eqref{eqn:cumb}. Hence no connected component of the graph \eqref{eqn:cumb} is isomorphic to $K_2$. We have managed to prove \eqref{eqn:notiso1}.

\subsection{Non-unigraphicity of $4^{k},3^4$ when $k\geq 3$}
\label{sec:4p}
To complete the proof of Theorem \ref{thm:1}, it remains to show that the degree sequence
\begin{equation}
	\label{eqn:s4}
	4^{k},3^4, \qquad k\geq 3
\end{equation}
has at least two non-isomorphic self-dual realisations. We will need the polyhedron $S(4,4)$ together with the vertex and face labellings given in Figure \ref{fig:44}, and also the following procedure.

\begin{alg}
	\label{alg:2}
	\
	
	\noindent\textbf{Input.} A polyhedral graph together with vertex and face labellings
	\[{v_1,v_2,\dots,v_p}, \qquad {f_1,f_2,\dots,f_r},\]
	such that $v_{p-1}v_p$ is an edge, and $v_{p-2}$ a vertex on the same face.
	
	\noindent\textbf{Output.} A plane graph together with vertex and region labellings
	\[{v_1,v_2,\dots,v_p,v_{p+1}}, \qquad {g_1,g_2,\dots,g_r,g_{r+1}}.\]
	
	\noindent\textbf{Description.} We remove the edge $v_{p-1}v_p$, and add an extra vertex $v_{p+1}$, together with new edges
	\[v_{p-2}v_{p+1}, \ v_{p-1}v_{p+1}, \ v_{p}v_{p+1}\]
	(this operation is an edge-splitting). Then, in the newly obtained plane graph, we relabel the regions as
	\[g_i:=f_{i-1}, \quad 3\leq i\leq p+1,\]
	$g_1:=[v_{p-1},v_p,v_{p+1}]$, and $g_2:=[v_{p-2},v_p,v_{p+1}]$.
\end{alg}

We apply Algorithm \ref{alg:2} repeatedly to $S(4,4)=:G_6$, letting
\[G_p, \quad p\geq 7, \quad p=k+4\]
be the resulting order $p$ graph, obtained after $p-6$ implementations. The first few iterations are depicted in Figure \ref{fig:sec4}.
\begin{figure}[h!]
	\centering
	\begin{subfigure}[m]{0.44\textwidth}
		\centering
		\includegraphics[width=5.5cm,clip=false]{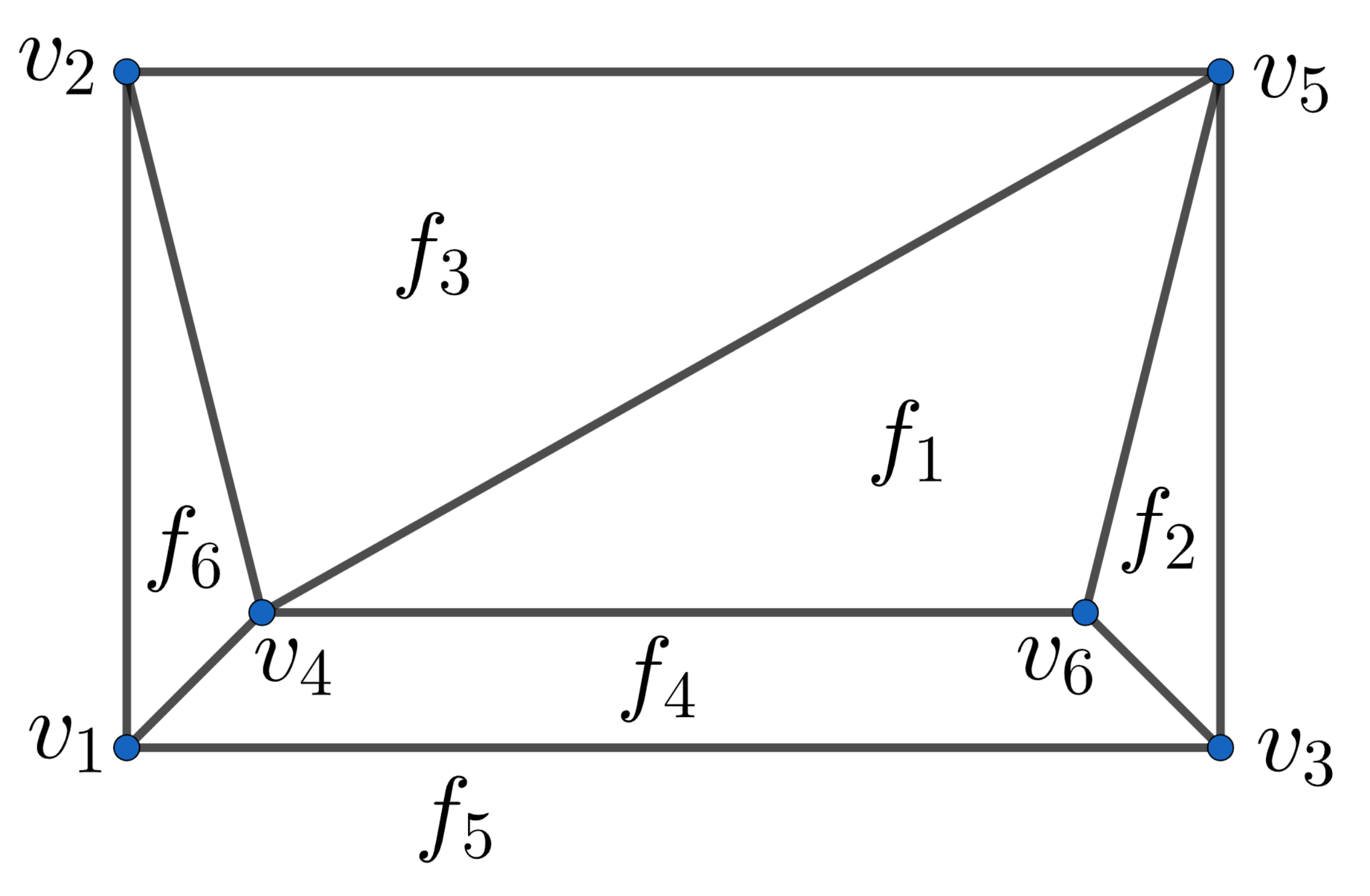}
		\caption{The initial $S(4,4)=G_6$.}
		\label{fig:G6}
	\end{subfigure}
	\hspace{0.5cm}
	\begin{subfigure}[m]{0.44\textwidth}
		\centering
		\includegraphics[width=5.5cm,clip=false]{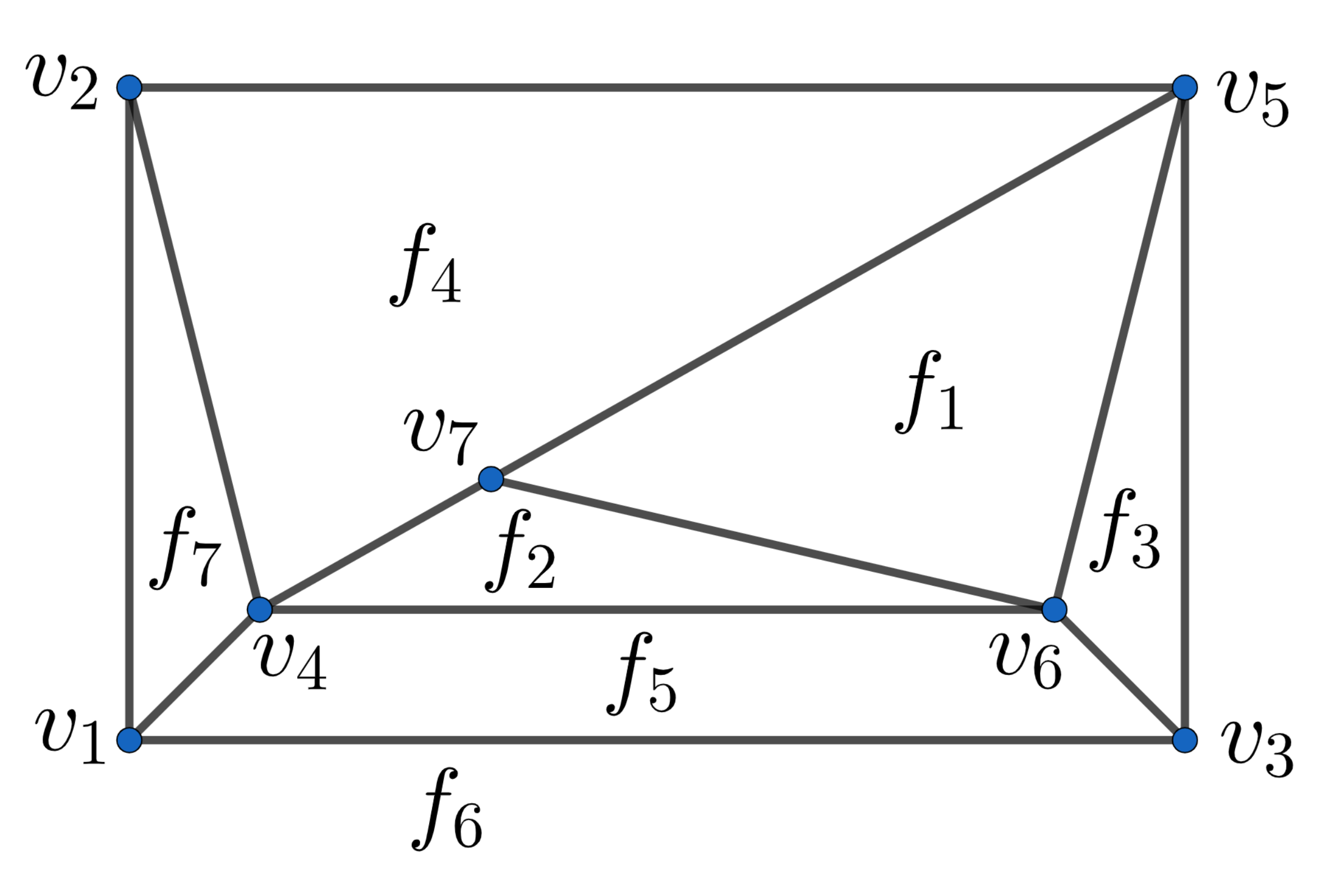}
		\caption{The graph $G_7$.}
		\label{fig:G7}
	\end{subfigure}
	\hspace{0.5cm}
	\begin{subfigure}[m]{0.44\textwidth}
		\centering
		\includegraphics[width=5.5cm,clip=false]{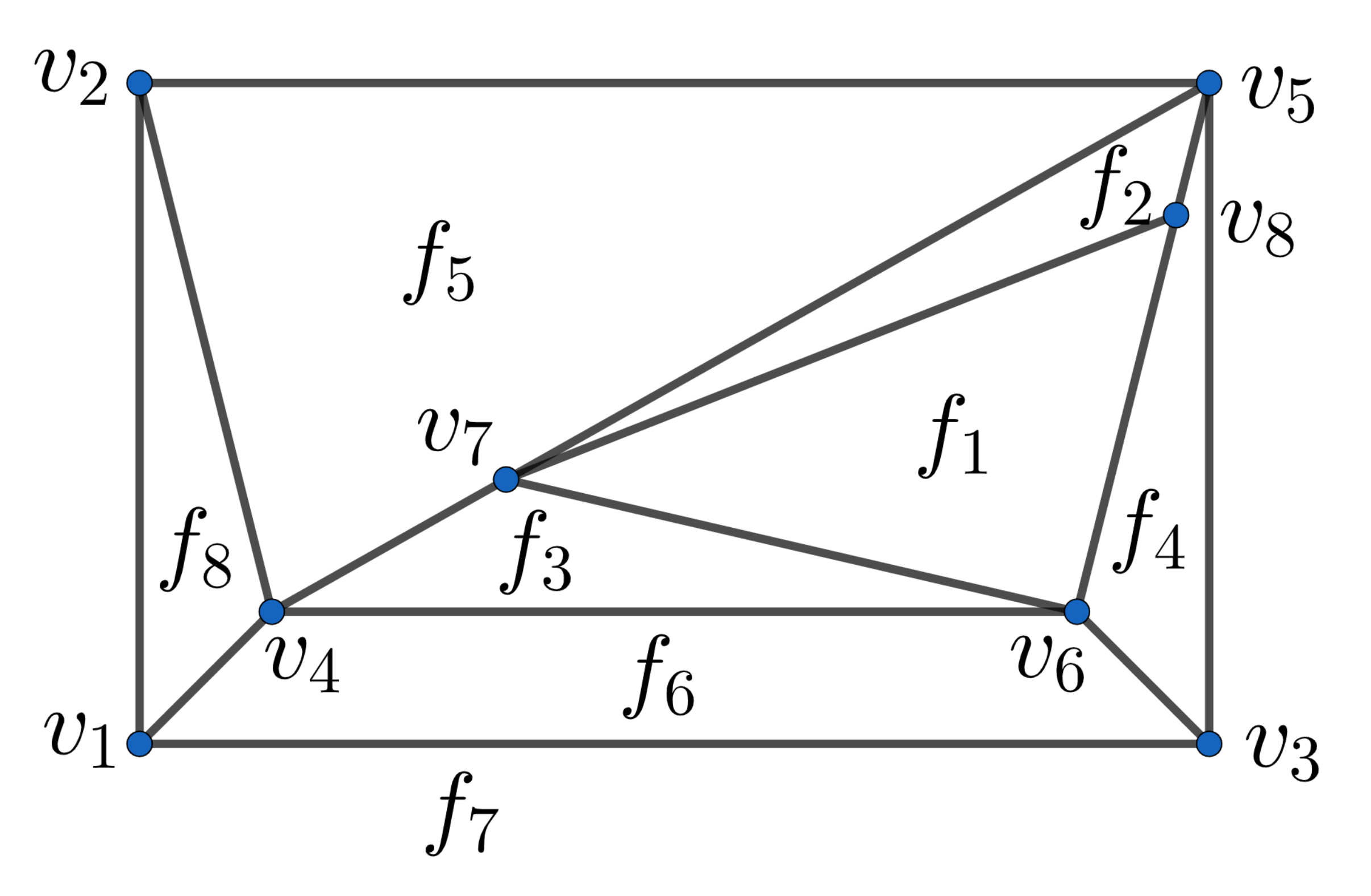}
		\caption{The graph $G_8$.}
		\label{fig:G8}
	\end{subfigure}
	\hspace{0.5cm}
	\begin{subfigure}[m]{0.44\textwidth}
		\centering
		\includegraphics[width=5.5cm,clip=false]{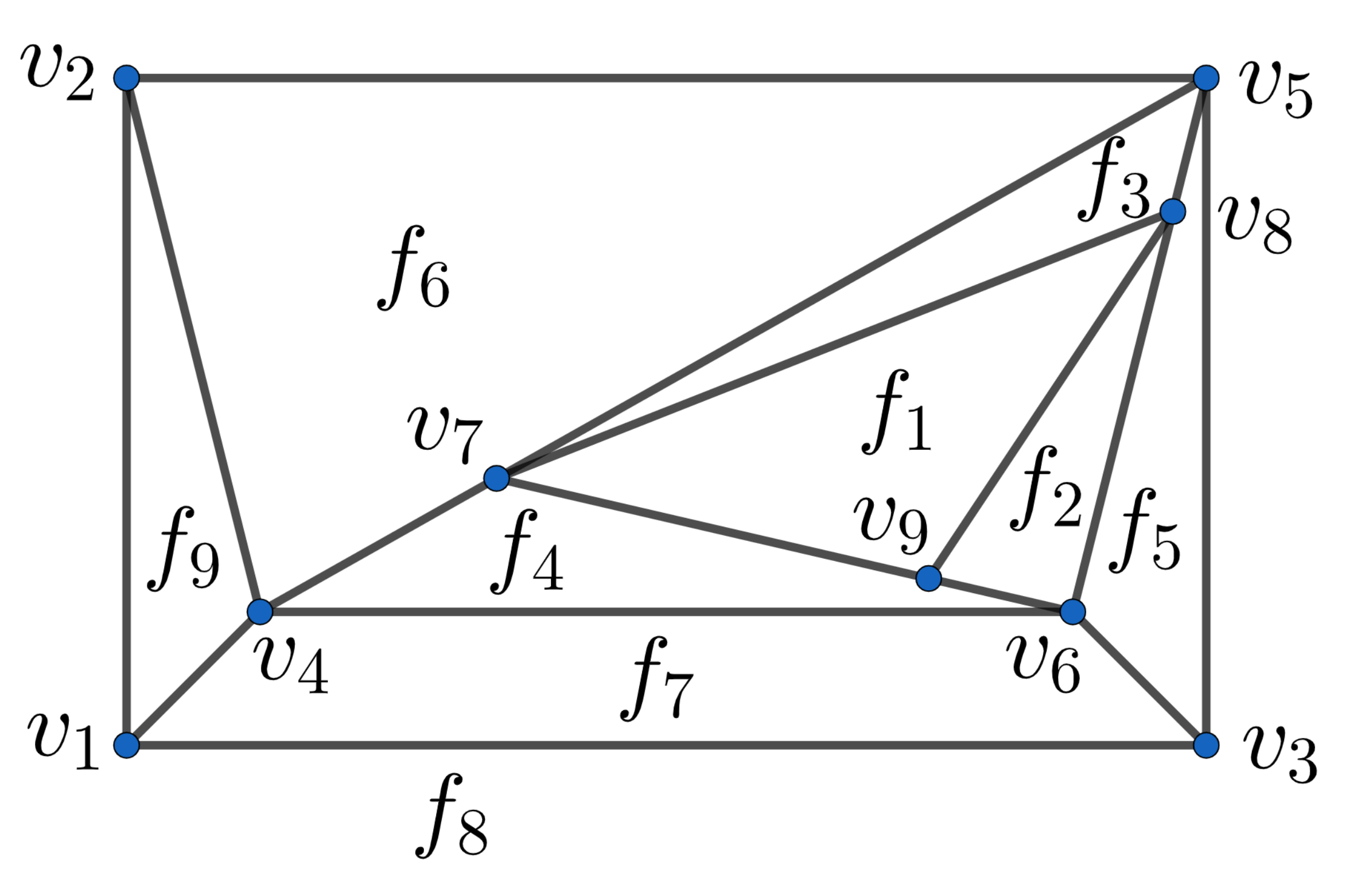}
		\caption{The graph $G_9$.}
		\label{fig:G9}
	\end{subfigure}
	\caption{The initial $S(4,4)$ and the first three iterations of Algorithm \ref{alg:2}.}
	\label{fig:sec4}
\end{figure}

We want to show that for every such $p$, $G_p$ is a self-dual $3$-polytope, of sequence \eqref{eqn:s4}, and moreover
\begin{equation}
	\label{eqn:notiso}
G_p\not\simeq P((\underbrace{4,4,\dots,4}_{p-4 \text{ times}})).
\end{equation}

First, starting with the vertex labelling of Figure \ref{fig:G6}, we have the face $[v_4,v_5,v_6]=f_1$, with $\deg(v_6)=3$. Each implementation of the algorithm increases the degree of $v_p$ by $1$, leaving other degrees unchanged, and also introduces the new $v_{p+1}$ of degree $3$. Further, $[v_{p-1},v_p,v_{p+1}]=g_1$ is a face in the new graph. The vertex/edge additions/deletions of Algorithm \ref{alg:2} amount to an edge-splitting operation on a polyhedron, that always produces another polyhedron. By induction, repeated implementation is thus well defined, and indeed produces $3$-polytopes of sequence \eqref{eqn:s4}.

Second, we will prove self-duality for each $G_p$ via induction. 
For clarity of exposition, we have chosen to sketch the sets of edges, rather than list them. Say that the set of edges of a particular self-dual polyhedron $G_p$ may be partitioned as in Figure \ref{fig:sec4f1}. By duality, the faces with corresponding indices are adjacent -- Figure \ref{fig:sec4f2}. These assumptions clearly hold for the initial $S(4,4)=G_6$. After the application of Algorithm \ref{alg:2} to $G_p$, the vertex relations for $G_{p+1}$ are given in Figure \ref{fig:sec4f3}. The corresponding face relations appear in Figure \ref{fig:sec4f4}, where we have not yet defined $g_1,\dots,g_{p+1}$. The new face $[v_{p-1},v_p,v_{p+1}]$ has been momentarily called $f_0$. 
After relabelling, the face relations for $G_{p+1}$ become those of Figure \ref{fig:sec4f5}. Comparing Figures \ref{fig:sec4f3} and \ref{fig:sec4f5}, we ascertain self-duality for $G_{p+1}$.

\begin{figure}[h!]
	\centering
	\begin{subfigure}[m]{0.99\textwidth}
		\centering
		\includegraphics[width=8.0cm,clip=false]{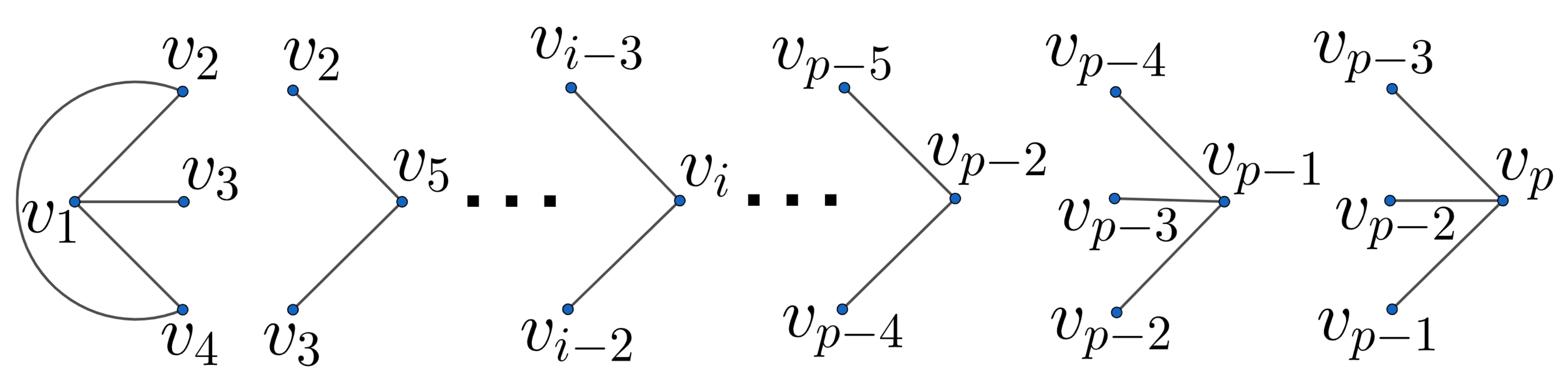}
		\caption{Initial vertex adjacencies for $G_p$.}
		\label{fig:sec4f1}
		\vspace{0.5cm}
	\end{subfigure}
	\begin{subfigure}[m]{0.99\textwidth}
		\centering
		\includegraphics[width=8.0cm,clip=false]{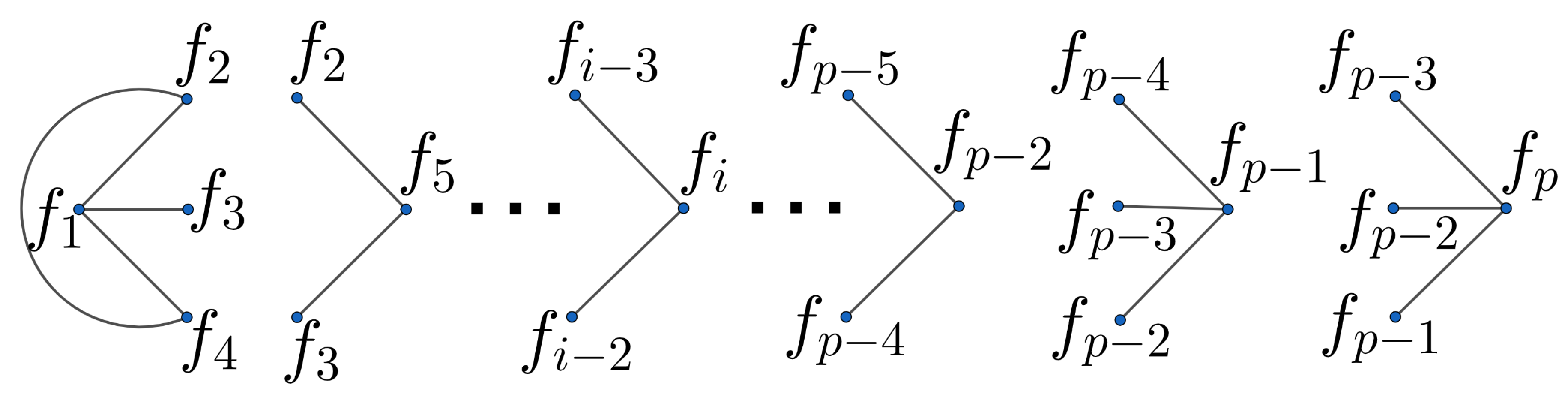}
		\caption{Initial face adjacencies for $G_p$.}
		\label{fig:sec4f2}
		\vspace{0.5cm}
	\end{subfigure}
	\begin{subfigure}[m]{0.99\textwidth}
		\centering
		\includegraphics[width=9.0cm,clip=false]{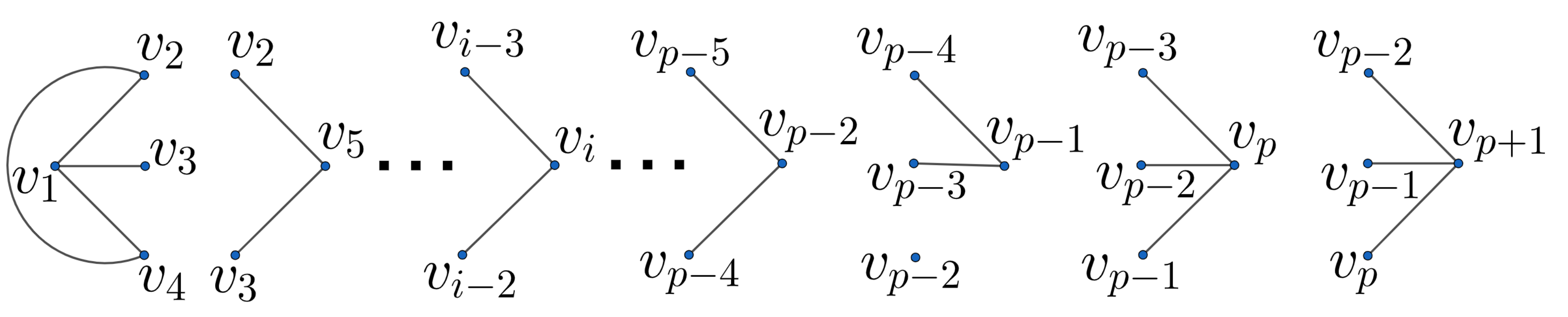}
		\caption{Vertex adjacencies in $G_{p+1}$.}
		\label{fig:sec4f3}
		\vspace{0.5cm}
	\end{subfigure}
	\begin{subfigure}[m]{0.99\textwidth}
		\centering
		\includegraphics[width=8.0cm,clip=false]{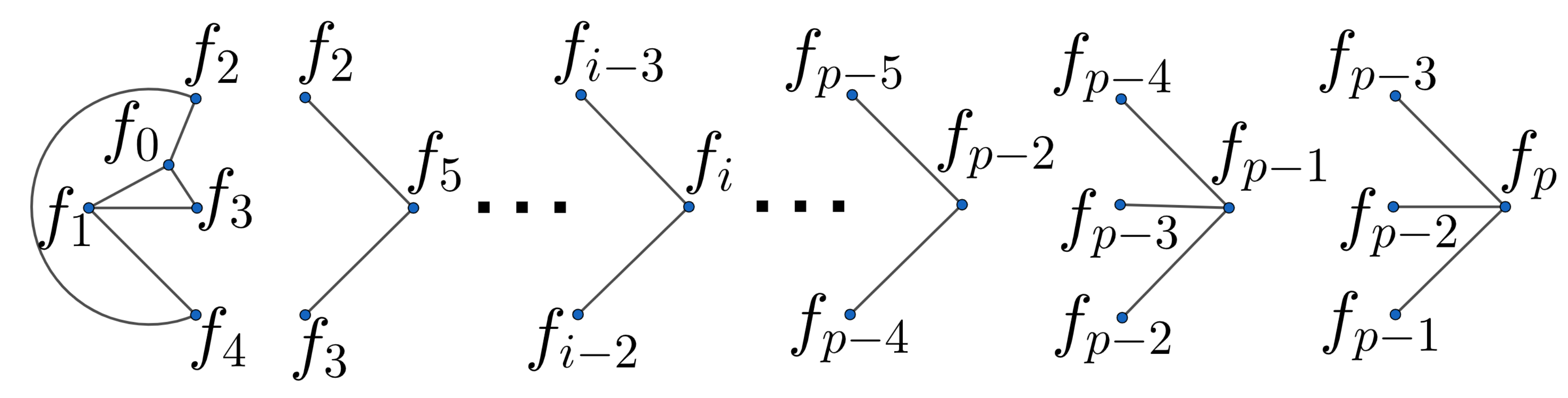}
		\caption{Face adjacencies in the transformed graph, before the relabelling, where $f_0:=[v_{p-1},v_p,v_{p+1}]$. 
		}
		\label{fig:sec4f4}
		\vspace{0.5cm}
	\end{subfigure}
	\begin{subfigure}[m]{0.99\textwidth}
		\centering
		\includegraphics[width=8.0cm,clip=false]{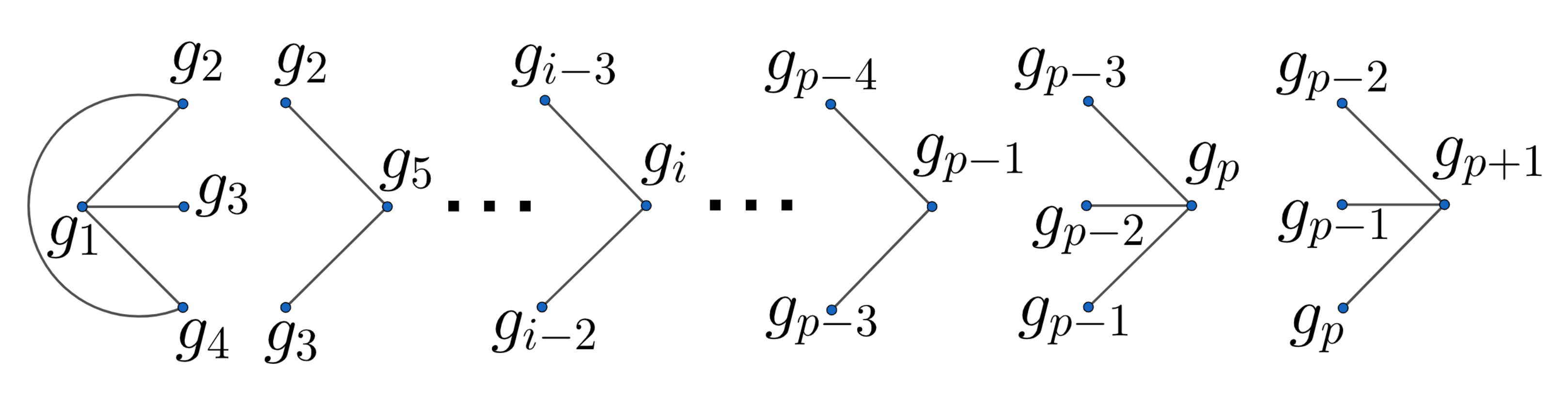}
		\caption{Face adjacencies in $G_{p+1}$ after relabelling.}
		\label{fig:sec4f5}
	\end{subfigure}
	\caption{If $G_p$ is self-dual, then so is $G_{p+1}$.}
	\label{fig:alg2}
\end{figure}

We have managed to prove that $G_p$ is a self-dual $3$-polytope, of sequence \eqref{eqn:s4}. We turn to showing \eqref{eqn:notiso}. Consider the subgraph $H_3(\Gamma)$ generated by the vertices of degree $3$ in a given graph $\Gamma$. By induction, in each $G_p$, the vertices of degree $3$ are $v_1,v_2,v_3,v_p$. Consulting Figure \ref{fig:sec4f1}, we see that
\begin{equation}
\label{eqn:H3Gp}
H_3(G_p)\simeq \calP_3\cup K_1, \quad p\geq 7,
\end{equation}
where $\calP_3$ is the elementary path on three vertices, and $\cup$ disjoint union of graphs.

On the other hand, for $p\geq 7$ we implement Algorithm \ref{alg:1} with input \[(\underbrace{4,4,\dots,4}_{p-4 \text{ times}})\]
on an initial cube. Recall the canonical labelling \eqref{eqn:can}. The vertices of the cube are
\[\{v_1,f_1,v_2,f_2,v_3,f_3,v_4,f_4\},\]
such that, referring to Figure \ref{fig:1a},
\[v_2=a,f_2=A,v_3=c,f_3=C,v_4=b,f_4=B,\]
with $v_1$ being adjacent to $f_1,f_2,f_3$. Algorithm \ref{alg:1} does not alter the adjacencies of $v_1,v_2$, or the face $[v_1,f_1,v_2,f_3]$. Then in the resulting
\begin{equation}
\label{eqn:cumb2}
P((\underbrace{4,4,\dots,4}_{p-4 \text{ times}})),
\end{equation}
$v_1,v_2$ are adjacent and of degree $3$.

Furthermore, in the algorithm, we apply $\calZ$ from Figure \ref{fig:1a} once, relabel as in Figure \ref{fig:1c}, and repeat a total of $p-4$ times. The newly inserted vertices are, in order, called
\[v_i,f_i, \qquad 5\leq i\leq p,\]
where the $v_i$'s are vertices of \eqref{eqn:cumb2}, and the $f_i$'s are vertices of the dual. After the $j$-th application of $\calZ$, the vertices labelled $c,b$ in Figure \ref{fig:1c}, now corresponding to $v_{j+3},v_{j+4}$ respectively, are adjacent and of degree $3$. This is true in particular for the last application of $\calZ$, hence $v_{p-1},v_p$ are adjacent and of degree $3$ in \eqref{eqn:cumb2}.

This is already enough to conclude that
\[H_3(P((\underbrace{4,4,\dots,4}_{p-4 \text{ times}})))\not\simeq \calP_3\cup K_1, \quad p\geq 7,\]
that together with \eqref{eqn:H3Gp} implies \eqref{eqn:notiso}. To be precise, as soon as $p\geq 7$, in the radial graph the vertices
\[v_{p-1},v_p\]
are separated from $v_1,v_2$ by the $6$-cycle
\[v_3,f_5,v_4,f_3,v_5,f_4.\]
Thereby, we actually have
\[H_3(P((\underbrace{4,4,\dots,4}_{p-4 \text{ times}})))\simeq K_2\cup K_2, \quad p\geq 7.\]
The proof of Theorem \ref{thm:1} is now complete.

\paragraph{Discussion and future research.} One feature of Algorithm \ref{alg:1} is that it allows to increase one vertex degree at a time, leaving the rest unchanged, and inserting new vertices only of degree $3$. Moreover, self-dual sequences have the nice property that, once the values greater than $3$ are fixed (i.e., the tuple $T$), the quantity of $3$'s $m(T)$ is determined also \eqref{eqn:m}. These facts make the class of self-dual polyhedra particularly versatile to work with in this context. It would be nice to apply the ideas of this paper to other subclasses of polyhedra, or better yet the full class -- Question \ref{que:gen}.
\\
On a different note, it would be interesting to modify Algorithm \ref{alg:1} to construct not only the self-dual polyhedral graph $P(T)$, but also the corresponding self-dual polyhedral solid embedded in $3$-dimensional space in such a way that it corresponds naturally to its dual e.g., via a polarity.

\paragraph{Statements and Declarations.}
The author has no relevant financial or non-financial interests to disclose. The are no conflicts of interests to disclose. The data produced to motivate this work is available on request.

\paragraph{Acknowledgements.}
The author wishes to thank two anonymous referees for their helpful comments on a previous version.
\\
Riccardo W. Maffucci was partially supported by Programme for Young Researchers `Rita Levi Montalcini' PGR21DPCWZ \textit{Discrete and Probabilistic Methods in Mathematics with Applications}, awarded to Riccardo W. Maffucci.

\bibliographystyle{abbrv}
\bibliography{bibgra}

\begin{thebibliography}{10}

\bibitem{arcric}
D.~Archdeacon and R.~B. Richter.
\newblock The construction and classification of self-dual spherical polyhedra.
\newblock {\em Journal of Combinatorial Theory, Series B}, 54(1):37--63, 1992.

\bibitem{bata89}
V.~Batagelj.
\newblock An inductive definition of the class of 3-connected quadrangulations
  of the plane.
\newblock {\em Discrete mathematics}, 78(1-2):45--53, 1989.

\bibitem{berjor}
A.~R. Berg and T.~Jord{\'a}n.
\newblock A proof of {C}onnelly's conjecture on 3-connected circuits of the
  rigidity matroid.
\newblock {\em Journal of Combinatorial Theory, Series B}, 88(1):77--97, 2003.

\bibitem{boro20}
E.~Boros, V.~Gurvich, M.~Milani{\v{c}}, and J.~Vi{\v{c}}i{\v{c}}.
\newblock On the degree sequences of dual graphs on surfaces.
\newblock {\em arXiv:2008.00573}, 2020.

\bibitem{bose08}
P.~Bose, V.~Dujmovi{\'c}, D.~Krizanc, S.~Langerman, P.~Morin, D.~R. Wood, and
  S.~Wuhrer.
\newblock A characterization of the degree sequences of 2-trees.
\newblock {\em Journal of Graph Theory}, 58(3):191--209, 2008.

\bibitem{bowe67}
R.~Bowen and S.~Fisk.
\newblock Generations of triangulations of the sphere.
\newblock {\em Mathematics of Computation}, 21(98):250--252, 1967.

\bibitem{br2005}
G.~Brinkmann, S.~Greenberg, C.~Greenhill, B.~D. McKay, R.~Thomas, and
  P.~Wollan.
\newblock Generation of simple quadrangulations of the sphere.
\newblock {\em Discrete mathematics}, 305(1-3):33--54, 2005.

\bibitem{brin07}
G.~Brinkmann, B.~D. McKay, et~al.
\newblock Fast generation of planar graphs.
\newblock {\em MATCH Commun. Math. Comput. Chem}, 58(2):323--357, 2007.

\bibitem{chro95}
M.~Chrobak and T.~H. Payne.
\newblock A linear-time algorithm for drawing a planar graph on a grid.
\newblock {\em Information Processing Letters}, 54(4):241--246, 1995.

\bibitem{de2024cancellation}
R.~De~March and R.~W. Maffucci.
\newblock Cancellation and regularity for planar, 3-connected {K}ronecker
  products.
\newblock {\em arXiv:2411.13473}.

\bibitem{delmaf}
J.~Delitroz and R.~W. Maffucci.
\newblock On unigraphic polyhedra with one vertex of degree $p-2$.
\newblock {\em Results in Mathematics}, 79(2):79, 2024.

\bibitem{dillen}
M.~B. Dillencourt.
\newblock Polyhedra of small order and their {H}amiltonian properties.
\newblock {\em Journal of combinatorial theory, Series B}, 66(1):87--122, 1996.

\bibitem{erdgal}
P.~Erd{\"o}s and T.~Gallai.
\newblock Graphs with prescribed degrees of vertices.
\newblock {\em Mat. Lapok}, 11:264--274, 1960.

\bibitem{grothendieck1984esquisse}
A.~Grothendieck.
\newblock {\em Esquisse d'un programme}.
\newblock 1984.

\bibitem{hakimi}
S.~Hakimi.
\newblock On the realizability of a set of integers as degrees of the vertices
  of a graph.
\newblock {\em SIAM Journal Applied Mathematics}, 1962.

\bibitem{haki06}
S.~L. Hakimi and E.~F. Schmeichel.
\newblock Graphs and their degree sequences: A survey.
\newblock In {\em Theory and Applications of Graphs: Proceedings, Michigan May
  11--15, 1976}, pages 225--235. Springer, 2006.

\bibitem{hare98}
D.~Harel and M.~Sardas.
\newblock An algorithm for straight-line drawing of planar graphs.
\newblock {\em Algorithmica}, 20:119--135, 1998.

\bibitem{hash11}
M.~Hasheminezhad, B.~McKay, T.~Reeves, et~al.
\newblock Recursive generation of simple planar 5-regular graphs and
  pentangulations.
\newblock {\em J. Graph Algorithms Appl.}

\bibitem{have55}
V.~Havel.
\newblock A remark on the existence of finite graphs.
\newblock {\em Casopis Pest. Mat.}, 80:477--480, 1955.

\bibitem{hollowbread2025generation}
P.~Hollowbread-Smith and R.~W. Maffucci.
\newblock Generation of 3-connected, planar line graphs.
\newblock {\em Discrete Mathematics}, 348(2):114302, 2025.

\bibitem{john80}
R.~Johnson.
\newblock Properties of unique realizations—a survey.
\newblock {\em Discrete Mathematics}, 31(2):185--192, 1980.

\bibitem{koren1}
M.~Koren.
\newblock Sequences with a unique realization by simple graphs.
\newblock {\em Journal of Combinatorial Theory, Series B}, 21(3):235--244,
  1976.

\bibitem{li1975}
S.-Y.~R. Li.
\newblock Graphic sequences with unique realization.
\newblock {\em Journal of Combinatorial Theory, Series B}, 19(1):42--68, 1975.

\bibitem{lova82}
L.~Lovasz and Y.~Yemini.
\newblock On generic rigidity in the plane.
\newblock {\em SIAM Journal on Algebraic Discrete Methods}, 3(1):91--98, 1982.

\bibitem{maffucci2024classification}
R.~W. Maffucci.
\newblock Classification and construction of planar, 3-connected {K}ronecker
  products.
\newblock {\em arXiv:2402.01407}.

\bibitem{maffucci2025regularity}
R.~W. Maffucci.
\newblock Regularity and separation for {S}ierpi{\'n}ski products of graphs.
\newblock {\em arXiv:2506.16864}.

\bibitem{mafpo2}
R.~W. Maffucci.
\newblock Constructing certain families of $3$-polytopal graphs.
\newblock {\em Journal of Graph Theory}, pages 1--18, 2022.

\bibitem{mafpo1}
R.~W. Maffucci.
\newblock On polyhedral graphs and their complements.
\newblock {\em Aequationes Mathematicae}, pages 1--15, 2022.

\bibitem{mafpo3}
R.~W. Maffucci.
\newblock Self-dual polyhedra of given degree sequence.
\newblock {\em Art Discrete Appl. Math. 6 (2023), P1.04.}, 2023.

\bibitem{maffucci2024characterising}
R.~W. Maffucci.
\newblock Characterising 3-polytopes of radius one with unique realisation.
\newblock {\em Australasian Journal of Combinatorics}, 89(2):268--293, 2024.

\bibitem{maffucci2024rao}
R.~W. Maffucci.
\newblock Rao's theorem for forcibly planar sequences revisited.
\newblock {\em Discrete Mathematics}, 347(10):114102, 2024.

\bibitem{maffucci2025faces}
R.~W. Maffucci.
\newblock On the faces of unigraphic 3-polytopes.
\newblock {\em European Journal of Combinatorics}, 124:104081, 2025.

\bibitem{mohtho}
B.~Mohar and C.~Thomassen.
\newblock {\em Graphs on surfaces}, volume~16.
\newblock Johns Hopkins University Press Baltimore, 2001.

\bibitem{rao978}
S.~Rao.
\newblock Characterization of forcibly planar degree sequences.
\newblock {\em ISI Tech. Report}, (36/78), 1978.

\bibitem{raosur}
S.~Rao.
\newblock A survey of the theory of potentially p-graphic and forcibly
  p-graphic degree sequences.
\newblock In {\em Combinatorics and Graph Theory: Proceedings of the Symposium
  Held at the Indian Statistical Institute, Calcutta, February 25--29, 1980},
  pages 417--440. Springer, 2006.

\bibitem{serchr}
B.~Servatius and P.~R. Christopher.
\newblock Construction of self-dual graphs.
\newblock {\em The American mathematical monthly}, 99(2):153--158, 1992.

\bibitem{serser}
B.~Servatius and H.~Servatius.
\newblock Self-dual maps on the sphere.
\newblock {\em Discret. Math.}, 134(1-3):139--150, 1994.

\bibitem{shabat2013drawing}
G.~B. Shabat and V.~A. Voevodsky.
\newblock Drawing curves over number fields.
\newblock In {\em The Grothendieck Festschrift: A Collection of Articles
  Written in Honor of the 60th Birthday of Alexander Grothendieck}, pages
  199--227. Springer, 2013.

\bibitem{radste}
E.~Steinitz and H.~Rademacher.
\newblock Vorlesungen {\"u}ber die {T}heorie der {P}olyeder. {S}pringer 1934.

\bibitem{tutt61}
W.~T. Tutte.
\newblock A theory of 3-connected graphs.
\newblock {\em Indag. Math}, 23(441-455):8, 1961.

\bibitem{tysh87}
R.~Tyshkevich, A.~Chernyak, and Z.~A. Chernyak.
\newblock Graphs and degree sequences. {I}.
\newblock {\em Cybernetics}, 23(6):734--745, 1987.

\bibitem{whit32}
H.~Whitney.
\newblock Congruent graphs and the connectivity of graphs.
\newblock {\em American Journal of Mathematics}, 54(1):150--168, 1932.

\end{thebibliography}
\end{document}